\documentclass[11pt,a4paper,reqno]{amsart}
\usepackage{graphicx}
\usepackage[linktocpage=true,colorlinks,citecolor=blue,linkcolor=blue,urlcolor=blue]{hyperref}
\usepackage{amssymb}
\usepackage{cite}
\usepackage{amsmath}
\usepackage{latexsym}
\usepackage{amscd}
\usepackage{amsthm}
\usepackage{mathrsfs}
\usepackage{url}
\newtheorem{thm}{Theorem}[section]
\newtheorem{corr}[thm]{Corollary}
\newtheorem{lem}[thm]{Lemma}
\newtheorem{prop}[thm]{Proposition}

\theoremstyle{definition}

\newtheorem{rem}[thm]{Remark}

\numberwithin{equation}{section}
\def\R{\mathbb R}

\def\ra{\rightarrow}
\def\pt{\partial}
\def\M{M}
\DeclareMathOperator\GG{G}
\DeclareMathOperator\GL{GL}
\DeclareMathOperator\SO{SO}

\DeclareMathOperator\vol{vol}

\DeclareMathOperator\inj{inj}

\begin{document}
\title[Laplacian flow for closed \texorpdfstring{$\GG_2$}{G2} structures]{Laplacian flow for closed \texorpdfstring{{\boldmath $\GG_2$}}{G2} structures: real analyticity}
\author{Jason D. Lotay}
\author{Yong Wei}
\address{Department of Mathematics, University College London, Gower Street, London, WC1E 6BT, United Kingdom}
\email{j.lotay@ucl.ac.uk; yong.wei@ucl.ac.uk}
\thanks{This research was supported by EPSRC grant EP/K010980/1.}

\begin{abstract}
Let $\varphi(t), t\in [0,T_0]$ be a solution to the Laplacian flow for closed $\GG_2$ structures on a compact $7$-manifold $\M$. We show that for each fixed time $t\in (0,T_0]$,  $(\M,\varphi(t),g(t))$ is  real analytic, where $g(t)$ is the  metric induced by $\varphi(t)$.  Consequently, any Laplacian soliton is real analytic and we obtain unique continuation results for the flow.
\end{abstract}
\maketitle
\tableofcontents

\section{Introduction}
Let $\M$ be a compact $7$-manifold and let $\varphi_0$ be a closed $\GG_2$ structure on $\M$. We consider solutions $\varphi(t), t\in [0,T_0]$, to the Laplacian flow  for closed $\GG_2$ structures:
\begin{equation}\label{Lap-flow-def}
  \left\{\begin{array}{rcl}
         \frac{\pt}{\pt t}\varphi &=&\Delta_{\varphi}\varphi,\\
           d\varphi &=&  0, \\
           \varphi(0) &=&\varphi_0,
         \end{array}\right.
\end{equation}
where $\Delta_{\varphi}\varphi=dd^*\varphi+d^*d\varphi$ is the Hodge Laplacian of $\varphi(t)$ with respect to the metric $g(t)$ determined by $\varphi(t)$.

The flow \eqref{Lap-flow-def} was introduced by Bryant (see \cite[\S 5]{bryant2005}) as a potential way to study the challenging problem of existence of
torsion-free $\GG_2$ structures and thus Ricci-flat metrics with exceptional holonomy $\GG_2$, since stationary
points of the flow are the $\GG_2$ structures $\varphi$ satisfying $d\varphi=d^*\varphi=0$, which is the
 torsion-free condition.  (Although this statement about the stationary points is true for compact manifolds by integration by parts, we gave an alternative argument in \cite{Lotay-Wei-1} which shows that stationary points of the flow are always torsion-free, even in the non-compact setting.)  Moreover, the flow moves within the cohomology
class of $\varphi_0$ and has a variational interpretation due to Hitchin \cite{hitchin2000, bryant-xu2011}. The primary goal in the field is to find conditions on an initial closed $\GG_2$ structure $\varphi_0$ such that the flow
\eqref{Lap-flow-def} will exist for all time and converge to a torsion-free $\GG_2$ structure.  Situations under which this occurs were proved by the authors in
\cite{Lotay-Wei-2}.

As $\M$ is compact, the Laplacian flow starting from any closed $\GG_2$ structure $\varphi_0$ is guaranteed to have a unique solution $\varphi(t)$ for a short time $t\in[0,\epsilon)$, where $\epsilon$ depends on $\varphi_0$ (see \cite{bryant2005,bryant-xu2011}). In our previous papers \cite{Lotay-Wei-1,Lotay-Wei-2}, we studied various foundational analytical and geometric properties of the flow \eqref{Lap-flow-def},
including Shi-type derivative estimates, uniqueness theorems, compactness results, soliton solutions, long-time existence results and stability of torsion-free $\GG_2$ structures along the flow.

On the face of it these analytic results are somewhat surprising because the velocity of the flow \eqref{Lap-flow-def} is defined by the Hodge Laplacian, which we would usually think
of as a positive operator, and thus the flow appears to look like a backwards heat equation.  In spite of this, the Laplacian flow is actually weakly parabolic
in a certain non-standard sense: it is parabolic in the direction of closed forms, modulo the action of diffeomorphisms.
It is this fact that enables the analysis of the flow to proceed. The reader is referred to \cite{Lotay-Wei-1,Lotay-Wei-2} for more detailed information about the Laplacian flow.

In this paper, we continue to analyze the Laplacian flow \eqref{Lap-flow-def} and investigate the regularity of
the solution $\varphi(t)$ for each positive time $t$.   Our main result is the following.

\begin{thm}\label{mainthm-local}
If $\varphi(t), t \in [0,T_0]$ is a smooth solution to the Laplacian flow \eqref{Lap-flow-def} for closed $\GG_2$ structures on an open set $U\subset \M$, then for each time $t\in (0,T_0]$, $(U,\varphi(t),g(t))$ is  real analytic.
\end{thm}

Readers are referred to \S \ref{sec:real-anal} for the definition and
 criterion for a $\GG_2$ structure to be real analytic.
Real analyticity for positive times is well known for
linear parabolic PDE (such as the heat equation) and some weakly parabolic nonlinear PDE (such as Ricci flow \cite{Bando}).  However, as we have indicated,
the Laplacian flow is not weakly parabolic in a standard manner, and so one should not immediately expect such a regularity result.

Since any Laplacian soliton corresponds to a local self-similar solution to the Laplacian flow \eqref{Lap-flow-def}, we have the following corollary to Theorem \ref{mainthm-local}.
\begin{corr}
Suppose $(\M, \varphi,X,\lambda)$ is a Laplacian soliton (not necessarily compact), i.e., $d\varphi=0$ and
\begin{equation}\label{solit}
  \Delta_{\varphi}\varphi=\lambda\varphi+\mathcal{L}_X\varphi
\end{equation}
for some smooth vector field $X$ and constant $\lambda$. Then $(\M,\varphi)$ is real analytic.
\end{corr}

The real analyticity of a torsion-free $\GG_2$ structure, i.e., the case $X=0$, $\lambda=0$ in \eqref{solit}, is already well-known (see \cite{bryant2010} for
example). Moreover, real analyticity plays a significant role in $\GG_2$ geometry, as can be seen in \cite{bryant2010}.

For convenience we say a $\GG_2$ structure $\varphi$ on $M$ is complete if its associated metric is complete. By modifying the argument in the proof of \cite[Corollary 6.4, p.256]{KoNo}, Theorem \ref{mainthm-local} immediately implies the following unique continuation results.
\begin{corr}
Suppose that $M^7$ is connected and simply connected, and $\varphi(t)$, $\tilde{\varphi}(t)$ are smooth
complete solutions to the Laplacian flow \eqref{Lap-flow-def} on $M\times [0, T_0]$. Then, for any $t\in (0, T_0]$, the following hold.
\begin{itemize}
\item[(a)] If $\varphi(t)=\tilde{\varphi}(t)$ on some open set $U\subset M$, then there exists a diffeomorphism $F$ of $M$ such that $F^*\tilde{\varphi}(t)\equiv\varphi(t) $.
\item[(b)] Any local diffeomorphism $F:U\ra V$ between connected open sets $U,V\subset M$ such that $F^*(\varphi(t)|_V)=\varphi(t)|_U$ can be uniquely extended to a global diffeomorphism $F$ of $M$ with $F^*\varphi(t)=\varphi(t)$.
\end{itemize}
 \end{corr}
\begin{corr}
Suppose that $M^7$ is connected and simply-connected and $(\varphi,X,\lambda)$ and $(\tilde{\varphi},\tilde{X},\tilde{\lambda})$ are complete Laplacian solitons on $M$.  If $\varphi=\tilde{\varphi}$ on some connected open set $U\subset M$, then there exists a diffeomorphism $F$ of $M$ such that $F^*\tilde{\varphi}\equiv\varphi$.
\end{corr}

Since a $\GG_2$ structure $\varphi$ determines a unique metric $g_{\varphi}$, any diffeomorphism $F: (M, \varphi)\ra (M,\tilde{\varphi})$ such that $F^*\tilde{\varphi}=\varphi$ is an isometry between $(M, g_{\varphi})$ and $(M,g_{\tilde{\varphi}})$. The converse is clearly not always true, since the $\GG_2$ structure encodes strictly more information than the metric.

Our approach to prove Theorem \ref{mainthm-local} is similar to Bando's \cite{Bando} proof of the real analyticity of Ricci flow, namely to use derivative estimates for the Riemann curvature tensor $Rm$, the torsion tensor $T$ and $\varphi$ along the flow.  In our previous paper \cite{Lotay-Wei-1}, we derived Shi-type derivative estimates along the Laplacian flow, which take the form
\begin{equation}\label{shi-0}
  t^{\frac k2}\left(|\nabla^kRm(x,t)|+|\nabla^{k+1}T(x,t)|\right)\leq C_kK,\quad x\in\M,\,t\in [0, 1/K],
\end{equation}
where $C_k$ is a constant depending on the order $k$ and $K$ is the bound on
 \begin{equation}\label{lambda-eq}
  \Lambda(x,t)=(|Rm|^2(x,t)+|\nabla T|^2(x,t))^{1/2}.
 \end{equation}
However, in \cite{Lotay-Wei-1}, we do not analyze how $C_k$ depends on $k$, which is particularly relevant when $k$ is large.

When one applies the heat operator to $ |\nabla^kRm(x,t)|+|\nabla^{k+1}T(x,t)|$,  lower order terms are generated during the computation, and the number of these terms grows with the order $k$ of differentiation, which then contributes to the growth of the constants $C_k$. By showing that the $C_k$ are of sufficiently slow growth in the order $k$, we may deduce that the $\GG_2$ structure $\varphi(t)$ and associated metric $g(t)$ are real analytic at each fixed time $t>0$.  The key step is to revisit the derivation of the derivative estimates \eqref{shi-0} from \cite{Lotay-Wei-1} and obtain the following much more refined estimates:
\begin{equation}\label{main-deri-1}
  \sum_{k=0}^n\frac {t^k}{(k+1)!^2}\left(|\nabla^kRm|^2(x,t)+|\nabla^{k+2}\varphi|^2(x,t)\right)\leq C(T_0,K_0)
\end{equation}
on $\M\times [0, \alpha/K_0]$ for all $n\in \mathbb{N}$ (we assume $\mathbb{N}$ to include $0$), where $K_0=\sup_\M|\Lambda(x,0)|$, $\alpha, C(T_0,K_0)$ are constants. As we will see in \S \ref{sec:real-anal}, the estimate \eqref{main-deri-1}  leads to the real analyticity of $(\M,\varphi(t),g(t))$ for each time $t>0$.

\section{Preliminaries} \label{sec:prelim}

We collect some facts on closed $\GG_2$ structures, mainly based on \cite{Lotay-Wei-1,bryant2005, Kar}.

Let $\{e_1,e_2,\cdots,e_7\}$ be the standard basis of $\R^7$ and let $\{e^1,e^2,\cdots,e^7\}$ be its dual basis. For simplicity we write $e^{ijk}=e^i\wedge e^j\wedge e^k$  and define a $3$-form $\phi$ by:
\begin{equation*}
  \phi=e^{123}+e^{145}+e^{167}+e^{246}-e^{257}-e^{347}-e^{356}.
\end{equation*}
The subgroup of $\GL(7,\R)$ fixing $\phi$ is the exceptional Lie group $\GG_2$, which is a compact, connected, simple Lie subgroup of $\SO(7)$ of dimension $14$.
It is well-known that $\GG_2$ acts irreducibly on $\R^7$ and preserves the metric and orientation for which $\{e_1,e_2,\cdots,e_7\}$ is an oriented orthonormal basis.

Let $\M$ be a $7$-manifold. For $x\in \M$ we let
\begin{equation*}
  \Lambda^3_+(\M)_x=\{\varphi_x\in\Lambda^3T_x^*\M\,|\,\exists u\in\textrm{Hom}(T_x\M,\R^7), u^*\phi=\varphi_x\}.
\end{equation*}
The bundle $\Lambda^3_+(\M)=\bigsqcup_x \Lambda^3_+(\M)_x$ is 
an open subbundle of $\Lambda^3T^*\M$.  We call a section $\varphi$ of $\Lambda^3_+(\M)$ a $\GG_2$ structure on $\M$ and denote the space of $\GG_2$ structures on $\M$ by $\Omega_+^3(\M)$.  The notation is motivated by the fact that there  is a 1-1 correspondence between $\GG_2$ structures in the sense of subbundles of the frame bundle and $\Omega^3_+(M)$.
The bundle $\Lambda^3_+(\M)$ has sections, which means that $\GG_2$ structures exist, if and only if $\M$ is oriented and spin.

A $\GG_2$ structure $\varphi$ induces a unique metric $g$ and orientation (given by a volume form $\vol_g$ of $g$) which satisfy
\begin{equation*}
  g(u,v)\vol_g=\frac 16(u\lrcorner\varphi)\wedge(v\lrcorner\varphi)\wedge\varphi.
\end{equation*}
 The metric and orientation determine the Hodge star operator $*_{\varphi}$, so we can define $\psi=*_{\varphi}\varphi$.
Notice that the relationship between $g$ and $\varphi$, and hence between $\psi$ and $\varphi$, is nonlinear.

Although $\GG_2$ acts irreducibly on $\R^7$ (and hence on $\Lambda^1(\R^7)^*$ and $\Lambda^6(\R^7)^*$), it acts reducibly on $\Lambda^k(\R^7)^*$ for $2\leq k\leq 5$. Hence a $\GG_2$ structure $\varphi$ induces splittings of the bundles $\Lambda^kT^*\M$ ($2\leq k\leq 5$), which we denote by $\Lambda^k_l(T^*\M)$ so that $l$ indicates the rank of the bundle, and we let the space of sections of $\Lambda^k_l(T^*\M)$ be $\Omega^k_l(\M)$.  Explicitly, we have that
\begin{align*}
  \Omega^2(\M)= & \Omega^2_7(\M)\oplus\Omega^2_{14}(\M) \quad \text{and} \quad
   \Omega^3(\M)= \Omega^3_1(\M)\oplus\Omega^3_7(\M)\oplus \Omega^3_{27}(\M),
\end{align*}
where (using the orientation in \cite{bryant2005} rather than \cite{Kar})
\begin{align*}
  \Omega^2_7(\M) &= \{\beta\in\Omega^2(\M)\,|\,\beta\wedge\varphi=2*_{\varphi}\beta\} =\{X\lrcorner\varphi\,|\,X\in C^{\infty}(T\M)\},\\
  \Omega^2_{14}(\M) &= \{\beta\in\Omega^2(\M)\,|\,\beta\wedge\varphi=-*_{\varphi}\beta\} =\{\beta\in\Omega^2(\M)\,|\,\beta\wedge\psi=0\},
\end{align*}
and
\begin{gather*}
  \Omega^3_1(\M)=\{f\varphi\,|\,f\in C^{\infty}(\M)\},\quad \Omega^3_7(\M)=\{X\lrcorner\psi\,|\,X\in C^{\infty}(T\M)\},\\
  \Omega^3_{27}(\M)=\{\gamma\in\Omega^3(\M)\,|\,\gamma\wedge\varphi=0=\gamma\wedge\psi\}.
\end{gather*}
Hodge duality gives corresponding decompositions of $\Omega^4(\M)$ and $\Omega^5(\M)$.

In our study it is convenient to write key quantities with respect to local coordinates $\{x^1,\cdots,x^7\}$ on $\M$. We write a $k$-form $\alpha$ locally as
\begin{equation*}
  \alpha=\frac 1{k!}\alpha_{i_1i_2\cdots i_k}dx^{i_1}\wedge\cdots\wedge dx^{i_k},
\end{equation*}
where $\alpha_{i_1i_2\cdots i_k}$ is totally skew-symmetric in its indices.
In particular, we write $\varphi$ and $\psi$ locally as
\begin{equation*}
  \varphi=\frac 16\varphi_{ijk}dx^i\wedge dx^j\wedge dx^k,\quad \psi=\frac 1{24}\psi_{ijkl}dx^i\wedge dx^j\wedge dx^k\wedge dx^l.
\end{equation*}

As in \cite{bryant2005} (up to a constant factor),  we define an operator
 $i_{\varphi}: S^2T^*\M\ra \Lambda^3T^*\M$ locally by
\begin{align*}
  (i_{\varphi}(h))_{ijk}&=h^l_i\varphi_{ljk}-h_j^l\varphi_{lik}-h_k^l\varphi_{lji}
\end{align*}
where $h=h_{ij}dx^idx^j$. Then $\Lambda^3_{27}(T^*\M)=i_{\varphi}(S^2_0T^*\M)$, where $S^2_0T^*\M$ denotes the bundle of trace-free symmetric $2$-tensors on $\M$, and $i_{\varphi}(g)=3\varphi$.

We have contraction identities for $\varphi$ and $\psi$ in index notation (see \cite{bryant2005,Kar}):
\begin{align}
  \varphi_{ijk}\varphi_{abl}g^{ia}g^{jb} &= 6g_{kl},\label{contr-iden-1}
\\
  \varphi_{ijq}\psi_{abkl}g^{ia}g^{jb}&= 4\varphi_{qkl},\label{contr-iden-2}
\displaybreak[0]\\
  \varphi_{ipq}\varphi_{ajk}g^{ia}&=g_{pj}g_{qk}-g_{pk}g_{qj}+\psi_{pqjk},\label{contr-iden-3}
\\
  \varphi_{ipq}\psi_{ajkl}g^{ia}&=g_{pj}\varphi_{qkl}-g_{jq}\varphi_{pkl}+g_{pk}\varphi_{jql}\nonumber\\
  &\qquad -g_{kq}\varphi_{jpl}+g_{pl}\varphi_{jkq}-g_{lq}\varphi_{jkp}.\label{contr-iden-4}
\end{align}

Given any $\GG_2$ structure $\varphi\in\Omega^3_+(\M)$, there exist unique differential forms (called the intrinsic torsion forms) $\tau_0\in\Omega^0(\M), \tau_1\in\Omega^1(\M), \tau_{2}\in\Omega^2_{14}(\M)$ and $\tau_{3}\in\Omega^3_{27}(\M)$ such that $d\varphi$ and $d\psi$ can be expressed as follows (see \cite{bryant2005}):
\begin{align*}
  d\varphi = \tau_0\psi+3\tau_1\wedge\varphi+*_{\varphi}\tau_{3}\quad\text{and}\quad
  d\psi =  4\tau_1\wedge\psi+\tau_{2}\wedge\varphi.
\end{align*}

We shall only consider closed $\GG_2$ structures $\varphi$ in this article.  In this case $d\varphi=0$ forces $\tau_0=\tau_1=\tau_3=0$, and hence the only
non-zero torsion form is $\tau_2$.  We therefore from now on set $\tau=\tau_2\in\Omega^2_{14}(\M)$ and reiterate that
\begin{equation*}
d\varphi=0\quad\text{and}\quad d\psi=\tau\wedge\varphi=-\!*_\varphi\!\tau.
\end{equation*}
We see immediately that
\begin{equation*}
  d^*\tau=*_{\varphi}d*_{\varphi}\tau 
=0,
\end{equation*}
which is given in local coordinates by $g^{mi}\nabla_m\tau_{ij}=0$.

The full torsion tensor is a $2$-tensor $T$ satisfying (see \cite{Kar})
\begin{gather} \label{nabla-var}
  \nabla_i\varphi_{jkl} =T_i^{\,\,m}\psi_{mjkl},\qquad
  T_i^{\,\,j}=\frac 1{24}\nabla_i\varphi_{lmn}\psi^{jlmn},\\
\label{nabla-psi}
  \nabla_m\psi_{ijkl} =-\Big( T_{mi}\varphi_{jkl}-T_{mj}\varphi_{ikl} -T_{mk}\varphi_{jil}-T_{ml}\varphi_{jki}\Big),
\end{gather}
where $T_{ij}=T(\pt_i,\pt_j)$ and $T_i^{\,\,j}=T_{ik}g^{jk}$.
In our setting we may compute that
\begin{equation*}
T=-\frac{1}{2}\tau,
\end{equation*}
so $T$ is divergence-free as $d^*\tau=0$.

Given these formulae we can compute the Hodge Laplacian of $\varphi$, which is the velocity of the Laplacian flow, as in \cite{bryant2005,Lotay-Wei-1}.

\begin{prop}\label{prop-hdl-var}
For a closed $\GG_2$ structure $\varphi$, the Hodge Laplacian of $\varphi$ satisfies
\begin{equation*}
  \Delta_{\varphi}\varphi=d\tau=i_{\varphi}(h)~\in~\Omega^3_1(\M)\oplus\Omega^3_{27}(\M),
\end{equation*}
where $h$ is a symmetric $2$-tensor on $\M$, locally given by
\begin{equation}\label{h-tensor-0}
h_{ij}=-\nabla_mT_{ni}\varphi_{j}^{\,\,mn}-\frac 13|T|^2g_{ij}-T_i^{\,\,l}T_{lj}.
\end{equation}
\end{prop}

Since  $\varphi$ determines a unique metric $g$ on $\M$, we then have the Riemann curvature tensor $Rm$ of $g$ on $\M$, which in our convention is given by
\begin{equation*}
  Rm(X,Y,Z,W)=g(\nabla_X\nabla_YZ-\nabla_Y\nabla_XZ-\nabla_{[X,Y]}Z,W)
\end{equation*}
for vector fields $X,Y,Z,W$ on $\M$. In local coordinates, we denote the components of $Rm$ by $R_{ijkl}=Rm(\pt_i,\pt_j,\pt_k,\pt_l)$. The Ricci curvature
 $Rc$ and scalar curvature $R$ are given locally by $R_{ik}=g^{jl}R_{ijkl}$ and $R=g^{ij}R_{ij}$,
and may be computed in terms of the torsion tensor as follows (see e.g.~\cite{Lotay-Wei-1}).
\begin{prop}\label{Ric-prop}
The Ricci tensor and the scalar curvature of the associated metric $g$ of a closed $\GG_2$ structure $\varphi$ are given as
\begin{gather}\label{Ricc-prop-closed}
R_{ik}=\nabla_jT_{li}\varphi_k^{\,\,\,jl}-T_{i}^{\,\,j}T_{jk}\quad\text{and}\quad
  R=-|T|^2=-T_{ik}T_{jl}g^{ij}g^{kl}.
\end{gather}
\end{prop}

Notice that $Rm$ and $\nabla T$ are both second order in $\varphi$, and $T$ is essentially $\nabla\varphi$, so
we might expect $Rm$ and $\nabla T$ to be related. The next proposition from \cite{Lotay-Wei-1} says that $\nabla T$ can be expressed using $T$ and $Rm$.
\begin{prop}\label{prop-nabla-T}
For a closed $\GG_2$ structure $\varphi$, we have
\begin{align*}
 2\nabla_iT_{jk} ~=&~ \frac 12R_{ijmn}\varphi_k^{\,\,\,mn}+\frac 12R_{kjmn} \varphi_i^{\,\,\,mn}-\frac 12R_{ikmn}\varphi_j^{\,\,\,mn}\\
 &~~-T_{im}T_{jn}\varphi_k^{\,\,\,mn} -T_{km}T_{jn}\varphi_i^{\,\,\,mn}+T_{im}T_{kn}\varphi_j^{\,\,\,mn}.
\end{align*}
\end{prop}

\section{Criterion for a \texorpdfstring{$\GG_2$}{G2} structure to be real analytic}\label{sec:real-anal}

Given a $7$-manifold $\M$, a real analytic structure on $\M$ is an atlas
\begin{equation*}
  \left\{(U_{j}, \{x_{j}^i\}_{i=1}^7)\right\}_{j\in J},
\end{equation*}
where $J$ is some indexing set, such that the transition functions are real analytic.  A Riemannian metric $g$ on a real analytic manifold $\M$ is then real analytic if the components $g_{ij}$ of $g$ are real analytic functions with respect to a subatlas of real analytic coordinates.

Let $\M$ be an orientable and spinnable $7$-manifold, let $\varphi$ be a $\GG_2$ structure on $\M$ and let $g$ be its associated Riemannian metric.
Suppose further that there is a subatlas of normal coordinate systems on $M$ such that the components of $g$ are real analytic functions in each of these
coordinate systems.
By \cite[Lemma 13.20]{Chow-etal-2008}, $(\M,g)$ is then a real analytic Riemannian manifold with respect to this subatlas and, in particular, an atlas for $\M$ can be found
with real analytic transition functions. In fact, by \cite[Lemma 1.2 \& Theorem 2.1]{DeT-Kaz1981}, for such $(\M,g)$ there exists an atlas of harmonic
coordinates  which are real analytic functions of the normal coordinates and so that the metric $g$ is real analytic in these harmonic coordinates.
Real analyticity of the transition functions for the atlas of harmonic coordinates then follows from the fact that the coordinates are harmonic and the metric is
 real analytic. If in addition the components $\varphi_{ijk}$ of $\varphi$ are real analytic with respect to the normal coordinates, which implies that $\varphi$ is also real analytic in the harmonic coordinates by \cite[Corollary 1.4]{DeT-Kaz1981}, then we say that $(\M,\varphi)$ is real analytic.

Let $\inj_g(p)$ denote the injectivity radius of $g$ at $p\in\M$, and if $\{x^i\}_{i=1}^7$ are coordinates centred at $p$ we denote the Christoffel symbols of the Levi-Civita connection of $g$ by $\Gamma^l_{ij}$ as usual and let
$\pt^k=\sum_{k_1+\ldots+k_n=k}\frac{\pt^k}{(\pt x^1)^{k_1}\cdots (\pt x^n)^{k_n}}$. With this notation in hand, we have the following derivative estimates of $\varphi$ and $g$ in normal coordinates.
\begin{lem}\label{lem-3-1}
Let $\varphi$ be a $\GG_2$ structure 
on $\M$ and let $g$ be its associated metric.  Let $p\in M$ and suppose there exist constants $C_1$ and $r>0$ such that
\begin{equation}\label{lem-3-1-condi1}
  |\nabla^kRm|(x)+|\nabla^{k+2}\varphi|(x)\leq C_1k!r^{-k-2}
\end{equation}
in a geodesic ball $B(p,r)$ for all $k\in \mathbb{N}$.
 
There exist constants $C_2, C_3, C_4, r_1=r_1(r), r_2=r_2(r)$ such that if we set $\rho=\min\{\frac{C_2}{\sqrt{C_1}}r, \inj_g(p)\}$ then, for all $x\in B(p,\rho)$ and $k\in\mathbb{N}$ we have in normal coordinates centred at $p$:
\begin{gather}
  \frac 12\delta_{ij} \leq  g_{ij}(x)\leq 2\delta_{ij},\qquad |\pt^kg_{ij}|(x)\leq C_3k!r_1^{-k}   \label{lem-3-1est0}\\
|\pt^{k}\Gamma_{ij}^l |(x)\leq C_3k!r_1^{-k-1},\quad |\pt^k\varphi_{ijl}|(x)\leq C_4k!r_2^{-k}.\label{lem-3-1est}
\end{gather} 
\end{lem}
\proof
The assumption \eqref{lem-3-1-condi1} implies  $|\nabla^kRm|(x)\leq C_1k!r^{-k-2}$ in $B(p,r)$.
The proof of  \cite[Corollary 4.12]{ha95-compact}
(see also \cite[Lemma 13.31]{Chow-etal-2008}) gives the existence of constants $C_2, C_3, r_1=r_1(r)>0$ such that for any $x\in B(p,\rho)$,
where $\rho$ is as stated, we have the derivative estimates for $g_{ij}$ and $\Gamma^l_{ij}$ in \eqref{lem-3-1est0}-\eqref{lem-3-1est}
 for all $k\in \mathbb{N}$.
 Thus it remains to show that, under the assumption \eqref{lem-3-1-condi1}, there are constants $C_4$ and $r_2(r)>0$ such that for all $k\in \mathbb{N}$ we have
\begin{equation*}
  |\pt^k\varphi_{ijl}|(x)\leq C_4k!r_2^{-k}.
\end{equation*}

In the following, we will prove a slightly stronger estimate:
\begin{equation}\label{lem-3-1pf3}
  |\pt^l\nabla^{k-l}\varphi|(x)\leq C_4k!2^{-(k-l)}r_2^{-k} 
\end{equation}
for all $0\leq l\leq k\leq m$, where $k,l,m\in \mathbb{N}$.  We prove \eqref{lem-3-1pf3} by induction on $m$. The case $m=0$ of \eqref{lem-3-1pf3} is trivial as $|\varphi|^2=7$. Suppose now that $m>1$ and \eqref{lem-3-1pf3} holds for all $0\leq l\leq k\leq m-1$.  We therefore only need to deal with the case where $k=m$ and we can perform an induction on $l$. Again, the case $k=m, l=0$ is trivial if we take $r_2\leq r/2$, as the condition \eqref{lem-3-1-condi1} gives that
\begin{equation*}
  |\nabla^m\varphi|\leq C_1(m-2)!r^{-m}\leq C_1m!r^{-m}\leq C_1m!2^{-m}r_2^{-m}.
\end{equation*}
So we now suppose that \eqref{lem-3-1pf3} holds for all $0\leq l<s$ for some $s\leq k=m$ and consider the case $l=s$. Since $\nabla^{(k-s)}\varphi$ is a $(k-s+3)$-tensor, we have
\begin{align*}
 |\pt^s&\nabla^{(k-s)}\varphi(x)|= ~\biggl|\pt^{(s-1)}\biggl(\nabla^{(k-s+1)}\varphi(x)+(k-s+3)\Gamma(x)*\nabla^{(k-s)}\varphi(x)\biggr)\biggr|\\
 \leq &~ \left|\pt^{(s-1)}\nabla^{(k-s+1)}\varphi(x)\right| \\&\quad
 +(k-s+3)\sum_{i=0}^{s-1}\binom{s-1}{i}\left|\pt^i\Gamma(x)\right|\left|\pt^{(s-1-i)}\nabla^{(k-s)}\varphi(x)\right|\displaybreak[0]\\
  \leq &~ C_4k!2^{-(k-s+1)}r_2^{-k} 
  \\
  &~+(k-s+3)2^{-(k-s)}C_3C_4\sum_{i=0}^{s-1}\binom{s-1}{i}i!(k-1-i)!r_1^{-i-1}
  r_2^{-(k-i-1)}
  \displaybreak[0]\\
  \leq &~ C_4k!2^{-(k-s)}r_2^{-k} 
  \biggl(\underbrace{\frac{1}{2} +(k-s+3)C_3\sum_{i=0}^{s-1}\frac{(s-1)!}{(s-1-i)!}\frac{(k-1-i)!}{k!}\frac{r_2^{1+i}}{r_1^{1+i}}}_{I}\biggr)
\end{align*}
To estimate the term $I$ in the bracket above,  by choosing $r_2\leq r_1$ 
 we have
\begin{align*}
  I\leq & ~ \frac 12+C_3\frac{k-s+3}{k}\frac{\binom{s-1}{i}}{\binom{k-1}{i}}\frac{r_2}{r_1}~\leq ~\frac 12+4C_3\frac{r_2}{r_1},
\end{align*}
as $k\geq 1$. Thus we can choose
\begin{equation*}
  r_2=r_2(r)=\min\left\{\frac{r_1}{8C_3},r_1,\frac{r}{2}\right\}>0
\end{equation*} 
such that $I\leq 1$ and then
\begin{align*}
 |\pt^s\nabla^{(k-s)}\varphi(x)|  \leq &~ C_4k!2^{-(k-s)}r_2^{-k}. 
\end{align*}
This completes the induction. 
\endproof

It is a routine exercise to show that \eqref{lem-3-1est0} and \eqref{lem-3-1est} imply that the coefficients of $g$ and $\varphi$ are real analytic with respect to normal coordinates.  Hence, by Lemma \ref{lem-3-1},  
if we have the derivative estimates \eqref{lem-3-1-condi1} for $Rm$ and $\varphi$, then we can conclude that $(\M,\varphi,g)$ is real analytic.

\section{Laplacian flow and evolution equations}\label{sec:evlution}

The goal of this paper is to prove the real analyticity of the solution to the Laplacian flow \eqref{Lap-flow-def}. From Proposition \ref{prop-hdl-var}, \eqref{Lap-flow-def} is equivalent to
\begin{equation}\label{flow-2}
  \frac{\pt}{\pt t}\varphi(t)=i_{\varphi}(h(t)),
\end{equation}
where $h(t)$ is the symmetric $2$-tensor on $\M$ given locally in \eqref{h-tensor-0}. By \eqref{Ricc-prop-closed}, we can also write $h$ locally as
\begin{equation}\label{h-tensor-1}
  h_{ij}=-R_{ij}-\frac 13|T|^2g_{ij}-2T_i^{\,\,k}T_{kj}.
\end{equation}
Notice that $T_i^{\,\,k}=T_{il}g^{kl}$ and $T_{il}=-T_{li}$.

Throughout the remainder of the article we will use the symbol $\Delta$ to denote the ``analyst's Laplacian'', which is a non-positive operator given in local
coordinates as $\nabla^i\nabla_i$, in contrast to the Hodge Laplacian $\Delta_{\varphi}$.

Under \eqref{flow-2},
the associated metric $g(t)$ of $\varphi(t)$ evolves by
\begin{equation}\label{evl-g}
  \frac{\pt}{\pt t}g(t)=2h(t).
\end{equation}
Substituting \eqref{h-tensor-1} into this equation, we have that 
\begin{equation}\label{flow-g2}
  \frac{\pt}{\pt t}g_{ij}=-2(R_{ij}+\frac 13|T|^2g_{ij}+2T_i^{\,\,k}T_{kj}).
\end{equation} Moreover, by \eqref{flow-g2}, the inverse of the metric evolves by
\begin{align}\label{flow-g^-1}
  \frac{\pt}{\pt t}g^{ij}=&-2h^{ij}=~2g^{ik}g^{jl}(R_{kl}+\frac 13|T|^2g_{kl}+2T_k^{\,\,m}T_{ml}).
\end{align}

The next lemma describes the evolution equations of the torsion tensor $T$, $\nabla\varphi$ and the curvature tensor $Rm$ along the Laplacian flow.  Here, and for the rest of the article, if $A,B$ are tensors and $k\in\mathbb{N}$, then $A*B$ denotes a contraction of tensors $A,B$ using only the metric $g$ (which is covariant constant) and we write a tensor $S\lessapprox kA*B$\footnote{Note that the inequality $``\lessapprox"$ can be differentiated, i.e., $\nabla S\lessapprox k\nabla A*B+kA*\nabla B$, unlike the usual inequality $``\leq"$ case.} if $S$ is equal to the sum of at most $k$ terms of the form $A*B$.
\begin{lem}
Suppose that $\varphi(t), t \in [0,T_0]$ is a solution to the Laplacian flow \eqref{Lap-flow-def} on a compact  manifold $\M$. The evolution equations of the torsion tensor $T$, $\nabla\varphi$ and the curvature tensor $Rm$ satisfy the following estimates:
\begin{align}
  \left(\frac{\pt}{\pt t}-\Delta\right)T\lessapprox &~8Rm*T+Rm*\nabla\varphi 
   +11\nabla T* T*\varphi+4 T*T*T;\label{evl-tors-0}\\
  \left(\frac{\pt}{\pt t}-\Delta\right)\nabla\varphi\lessapprox & ~62\nabla T*T*\varphi+6\nabla T*\nabla\varphi*\varphi+29 Rm*\nabla\varphi+Rm*T\nonumber\\
  &+Rm*\varphi*(T*\varphi+\nabla\varphi*\varphi)+24 T*T*\nabla\varphi\label{evl-nab-var-0}
\intertext{and}
\label{evl-Rm-0}
  \left(\frac{\pt}{\pt t}-\Delta\right)Rm\lessapprox & ~ 33 Rm*Rm+4Rm*T*T +35 \nabla\left(\nabla T*T\right).
\end{align}
\end{lem}
\proof
The estimates \eqref{evl-tors-0} and \eqref{evl-Rm-0} follow directly from the evolution equations of $T$ and $Rm$ along the Laplacian flow, which have been derived in \cite[\S 3]{Lotay-Wei-1}. 
To show \eqref{evl-nab-var-0}, recall that $\nabla\varphi$ and $T$ are related by
\begin{equation*}
  \nabla_i\varphi_{jkl}=T_{im}g^{mn}\psi_{njkl}.
\end{equation*}
Then we have
\begin{align}\label{est-I}
   \frac{\pt}{\pt t} \nabla_i\varphi_{jkl}=&  \left(\frac{\pt}{\pt t}T_{im}\right)g^{mn}\psi_{njkl}+T_{im}\left(\frac{\pt}{\pt t}g^{mn}\right)\psi_{njkl}+ T_{im}g^{mn} \left(\frac{\pt}{\pt t} \psi_{njkl}\right)\nonumber\\
   =&~I+II+III.
\end{align}
For the first term $I$, recall that from \cite[\S 3.2]{Lotay-Wei-1}, we have
\begin{align*}
 \frac{\pt}{\pt t}T_{ij} =&\Delta T_{ij}+R_i^{\,\,k}T_{kj}+\frac 12R_{ijmk}T^{mk}+\frac 12R_{mpi}^{\quad\,\,\, k}\nabla_k\varphi_j^{\,\,\,pm}\nonumber\\
  &+\nabla_pT_{qi}\nabla_m\varphi_n^{\,\,\,pq}\varphi_j^{\,\,mn}+\frac 13\nabla_m|T|^2\varphi_{ij}^{\,\,\,\,\, m}\nonumber\\
  &+\nabla_m(T_i^{\,\,k}T_{kn})\varphi_j^{\,\,mn}-T_i^{\,\,k}\nabla_pT_{qk}\varphi_j^{\,\,pq}-\frac 13|T|^2T_{ij}-T_i^{\,\,k}T_k^{\,\,m}T_{mj}.
 \end{align*}
Then
\begin{align}\label{est-I1}
  I=~ & \left(\frac{\pt}{\pt t}T_{im}\right)g^{mn}\psi_{njkl} \nonumber\displaybreak[0]\\
   =~&\Delta T_{im}g^{mn}\psi_{njkl}+R_i^{\,\,p}T_{pm}g^{mn}\psi_{njkl}+\frac 12R_{impq}T^{pq}g^{mn}\psi_{njkl}\nonumber\displaybreak[0]\\
   &+\frac 12R_{pqi}^{\quad\,\,s}\nabla_s\varphi_m^{\,\,\,\, qp}g^{mn}\psi_{njkl}+\nabla_pT_{qi}\nabla_s\varphi_t^{\,\,\,\,pq}\varphi_m^{\,\,st}g^{mn}\psi_{njkl}\nonumber\\
  &+\frac 13\nabla_p|T|^2\varphi_{im}^{\,\,\,\,\, \, p}g^{mn}\psi_{njkl}+\nabla_p(T_i^{\,\,s}T_{sq})\varphi_m^{\,\,\,pq}g^{mn}\psi_{njkl}\nonumber\\
  &-T_i^{\,\,s}\nabla_pT_{qs}\varphi_m^{\,\,\,pq}g^{mn}\psi_{njkl}-\frac 13|T|^2T_{im}g^{mn}\psi_{njkl}\nonumber\\
  &-T_i^{\,\,p}T_p^{\,\,q}T_{qm}g^{mn}\psi_{njkl}.
 \end{align}
Using \eqref{nabla-psi},
\begin{align*}
  \Delta  &\nabla_i\varphi_{jkl}=  \Delta\left(T_{im}g^{mn}\psi_{njkl}\right) \nonumber\\
  = &\left(\Delta T_{im}\right)g^{mn}\psi_{njkl}+T_{im}g^{mn}\Delta\psi_{njkl}+2\left(\nabla_pT_{im}\right)g^{mn}g^{pq}\nabla_q\psi_{njkl} \nonumber\displaybreak[0]\\
  =&\left(\Delta T_{im}\right)g^{mn}\psi_{njkl}-T_{i}^{\,\,\,n}\nabla^p\left(T_{pn}\varphi_{jkl}+T_{pj}\varphi_{nkl}+T_{pk}\varphi_{jnl}+T_{pl}\varphi_{jkn}\right)\nonumber\\
  &-2\left(\nabla_pT_{im}\right)g^{mn}g^{pq}\left(T_{qn}\varphi_{jkl}+T_{qj}\varphi_{nkl}+T_{qk}\varphi_{jnl}+T_{ql}\varphi_{jkn}\right)\nonumber\\
  =&\left(\Delta T_{im}\right)g^{mn}\psi_{njkl}-T_{i}^{\,\,\,n}(T_{pn}\nabla^p\varphi_{jkl}+T_{pj}\nabla^p\varphi_{nkl}+T_{pk}\nabla^p\varphi_{jnl}\nonumber\\
  &+T_{pl}\nabla^p\varphi_{jkn})-2\nabla_pT_{i}^{\,\,\,n}g^{pq}\left(T_{qn}\varphi_{jkl}+T_{qj}\varphi_{nkl}+T_{qk}\varphi_{jnl}+T_{ql}\varphi_{jkn}\right),
\end{align*}
where in the last equality we used $\nabla^pT_{pk}=0$. Using \eqref{nabla-var}, the second term of \eqref{est-I1} is equal to $R_i^{\,\,p}\nabla_p\varphi_{jkl}$ and the last two terms of \eqref{est-I1} can be rewritten as
\begin{align*}
  -\frac 13|T|^2\nabla_i\varphi_{jkl}-T_i^{\,\,p}T_p^{\,\,q}\nabla_q\varphi_{jkl}.
\end{align*}
The third and fourth terms of \eqref{est-I1} can be expressed using the contraction identity \eqref{contr-iden-3} as follows:
\begin{align*}
 & \frac 12R_{impq}T^{pq}g^{mn}\psi_{njkl}+\frac 12R_{pqi}^{\quad\,\,s}\nabla_s\varphi_m^{\,\,\,\, qp}g^{mn}\psi_{njkl} \\
  = & \frac 12R_{impq}T^{pq}g^{mn}\left(\varphi_{snj}\varphi_{tkl}g^{st}-g_{nk}g_{jl}+g_{nl}g_{jk}\right)\\
  &+\frac 12R_{pqi}^{\quad\,\,s}\nabla_s\varphi_m^{\,\,\,\, qp}g^{mn}\left(\varphi_{snj}\varphi_{tkl}g^{st}-g_{nk}g_{jl}+g_{nl}g_{jk}\right)\\
  =&\frac 12R_{impq}T^{pq}g^{mn}\varphi_{snj}\varphi_{tkl}g^{st}+\frac 12R_{pqi}^{\quad\,\,s}\nabla_s\varphi_m^{\,\,\,\, qp}g^{mn}\varphi_{snj}\varphi_{tkl}g^{st}\\
  &-\frac 12R_{ikpq}T^{pq}g_{jl}+\frac 12R_{ilpq}T^{pq}g_{jk}+\frac 12R_{pqi}^{\quad\,\,s}\nabla_s\varphi_k^{\,\,\,\, qp}g_{jl}+\frac 12R_{pqi}^{\quad\,\,s}\nabla_s\varphi_l^{\,\,\,\, qp}g_{jk}.
\end{align*}
Thus, we obtain our expression for the first term $I$ in \eqref{est-I}:
 \begin{align}\label{est-I4}
  I=&\Delta  \nabla_i\varphi_{jkl}+2\nabla_mT_i^{\,\,\,n}\left(T_{mn}\varphi_{jkl}+T_{mj}\varphi_{nkl}+T_{mk}\varphi_{jnl}+T_{ml}\varphi_{jkn}\right)\nonumber\displaybreak[0]\\
  &+T_{i}^{\,\,\,n}\left(T_{pn}\nabla^p\varphi_{jkl}+T_{pj}\nabla^p\varphi_{nkl}+T_{pk}\nabla^p\varphi_{jnl}+T_{pl}\nabla^p\varphi_{jkn}\right)
  +R_i^{\,\,p}\nabla_p\varphi_{jkl}\nonumber\\&+\frac 12R_{impq}T^{pq}g^{mn}\varphi_{snj}\varphi_{tkl}g^{st}+\frac 12R_{pqi}^{\quad\,\,s}\nabla_s\varphi_m^{\,\,\,\, qp}g^{mn}\varphi_{snj}\varphi_{tkl}g^{st}\nonumber\displaybreak[0]\\
   &-\frac 12R_{ikpq}T^{pq}g_{jl}+\frac 12R_{ilpq}T^{pq}g_{jk}+\frac 12R_{pqi}^{\quad\,\,s}\nabla_s\varphi_k^{\,\,\,\, qp}g_{jl}\nonumber\displaybreak[0]\\
   &+\frac 12R_{pqi}^{\quad\,\,s}\nabla_s\varphi_l^{\,\,\,\, qp}g_{jk}+\nabla_pT_{qi}\nabla_s\varphi_t^{\,\,\,\,pq}\varphi_m^{\,\,st}g^{mn}\psi_{njkl}\nonumber\displaybreak[0]\\
  &+\frac 13\nabla_p|T|^2\varphi_{im}^{\,\,\,\,\, \, p}g^{mn}\psi_{njkl}+\nabla_p(T_i^{\,\,s}T_{sq})\varphi_m^{\,\,\,\,pq}g^{mn}\psi_{njkl}\nonumber\\
  &-T_i^{\,\,s}\nabla_pT_{qs}\varphi_m^{\,\,\,pq}g^{mn}\psi_{njkl}-\frac 13|T|^2\nabla_i\varphi_{jkl}-T_i^{\,\,p}T_p^{\,\,q}\nabla_q\varphi_{jkl}.
\end{align}
Here we leave the terms involving $\varphi_{im}^{\,\,\,\,\, \, p}g^{mn}\psi_{njkl}$ and related expressions unchanged in \eqref{est-I4}, but observe that they can be expressed in terms of $\varphi$ using the contraction identity \eqref{contr-iden-4}.

The second term $II$ in \eqref{est-I} can be estimated using \eqref{flow-g^-1}. For the third term $III$ in \eqref{est-I}, recall from the contraction identity \eqref{contr-iden-3}, we have
\begin{equation*}
  \psi_{ijkl}=\varphi_{mij}\varphi_{nkl}g^{mn}-g_{ik}g_{jl}+g_{il}g_{jk}.
\end{equation*}
By \eqref{flow-2}, \eqref{flow-g2} and \eqref{flow-g^-1}, we can then derive that (see e.g.~\cite{Kar})
\begin{equation}\label{evl-psi}
  \frac{\pt}{\pt t}\psi_{ijkl}=h_i^m\psi_{mjkl}+h_j^m\psi_{imkl}+h_k^m\psi_{ijml}+h_l^m\psi_{ijkm},
\end{equation}
where $h$ is given in \eqref{h-tensor-0} (and equivalently in \eqref{h-tensor-1}). Then using \eqref{flow-g^-1} and \eqref{evl-psi}, we have that
$II+III$ is equal to
\begin{align}\label{est-II1}
  -2T_{im}h^{mn}&\psi_{njkl}+ T_{im}g^{mn} \big(h_n^s\psi_{sjkl}+h_j^s\psi_{nskl}   +h_k^s\psi_{njsl}+h_l^s\psi_{njks}\big)\nonumber\\
   =&~T_{im}g^{mn} \big(-h_n^s\psi_{sjkl}+h_j^s\psi_{nskl}+h_k^s\psi_{njsl}+h_l^s\psi_{njks}\big).
\end{align}
By \eqref{contr-iden-4} and \eqref{h-tensor-0}, the first term on the right-hand side of \eqref{est-II1} is
\begin{align*}
  -T_{im}g^{mn}h_n^s\psi_{sjkl} &=T_{im}g^{mn}\left(\nabla_pT_{qn}\varphi^{spq}+\frac 13|T|^2\delta_{ns}+T_n^{\,\,\,p}T_p^{\,\,\,s}\right)\psi_{sjkl} \nonumber\\
   &=T_i^{\,\,\,n}\nabla_jT_{qn}\varphi_{kl}^{\,\,\,\,\,q}-T_i^{\,\,\,n}\nabla_pT_{jn}\varphi_{kl}^{\,\,\,\,\,p}-T_i^{\,\,\,n}\nabla_kT_{qn}\varphi_{jl}^{\,\,\,\,\,q}\nonumber\\
   &+T_i^{\,\,\,n}\nabla_pT_{kn}\varphi_{jl}^{\,\,\,\,\,q}+T_i^{\,\,\,n}\nabla_lT_{qn}\varphi_{jk}^{\,\,\,\,\,q}-T_i^{\,\,\,n}\nabla_pT_{ln}\varphi_{jk}^{\,\,\,\,\,p}\nonumber\\
   &+\frac 13|T|^2\nabla_i\varphi_{jkl}+T_i^{\,\,\,n}T_n^{\,\,\,p}\nabla_p\varphi_{jkl}.
\end{align*}
By \eqref{nabla-var} and \eqref{h-tensor-1}, the remaining three terms of \eqref{est-II1} are equal to
\begin{align*}
   T_{im}g^{mn} \big(h_j^s\psi_{nskl}&+h_k^s\psi_{njsl}+h_l^s\psi_{njks}\big)
   =h_j^s\nabla_i\varphi_{skl}+h_k^s\nabla_i\varphi_{jsl}+h_l^s\nabla_i\varphi_{jks}\nonumber\\
   =&-\left(R_{jp}+\frac 13|T|^2g_{jp}+2T_j^{\,\,\,s}T_{kp}\right)g^{pq}\nabla_i\varphi_{qkl}\nonumber\\
    &-\left(R_{kp}+\frac 13|T|^2g_{kp}+2T_k^{\,\,\,s}T_{kp}\right)g^{pq}\nabla_i\varphi_{jql}\nonumber\\
    &-\left(R_{lp}+\frac 13|T|^2g_{lp}+2T_l^{\,\,\,s}T_{kp}\right)g^{pq}\nabla_i\varphi_{jkq}.
\end{align*}
Therefore,
\begin{align}\label{est-II3}
 II+III= &\, T_i^{\,\,\,n}\nabla_jT_{qn}\varphi_{kl}^{\,\,\,\,\,q}-T_i^{\,\,\,n}\nabla_pT_{jn}\varphi_{kl}^{\,\,\,\,\,p}-T_i^{\,\,\,n}\nabla_kT_{qn}\varphi_{jl}^{\,\,\,\,\,q}\nonumber\\
   &+T_i^{\,\,\,n}\nabla_pT_{kn}\varphi_{jl}^{\,\,\,\,\,q}+T_i^{\,\,\,n}\nabla_lT_{qn}\varphi_{jk}^{\,\,\,\,\,q}-T_i^{\,\,\,n}\nabla_pT_{ln}\varphi_{jk}^{\,\,\,\,\,p}\nonumber\\
   &+\frac 13|T|^2\nabla_i\varphi_{jkl}+T_i^{\,\,\,n}T_n^{\,\,\,p}\nabla_p\varphi_{jkl} \nonumber\displaybreak[0]\\
  & -\left(R_{jp}+\frac 13|T|^2g_{jp}+2T_j^{\,\,\,s}T_{kp}\right)g^{pq}\nabla_i\varphi_{qkl}\nonumber\\
    &-\left(R_{kp}+\frac 13|T|^2g_{kp}+2T_k^{\,\,\,s}T_{kp}\right)g^{pq}\nabla_i\varphi_{jql}\nonumber\\
    &-\left(R_{lp}+\frac 13|T|^2g_{lp}+2T_l^{\,\,\,s}T_{kp}\right)g^{pq}\nabla_i\varphi_{jkq}.
\end{align}
The estimate \eqref{evl-nab-var-0} then follows from \eqref{est-I}, \eqref{est-I4} and \eqref{est-II3}.
\endproof

\begin{rem}
Although $\nabla\varphi$ can be expressed using $T$  via \eqref{nabla-var}, it is not straightforward to write $\nabla^k\varphi$ in terms of $\nabla^jT, j=0,1,\cdots,k-1$. The evolution equation \eqref{evl-nab-var-0} for $\nabla\varphi$  is thus useful in $\S$\ref{sec:global} to estimate $\nabla^k\varphi$.
\end{rem}

\section{Global real analyticity}\label{sec:global}

In this section, we first prove the key derivative estimates for $Rm,T$ and $\varphi$, and then deduce Theorem \ref{mainthm-local} in the special case when $U=M$ is compact.

\subsection{Commutator formula}
First, we have the following commutator formula for $\nabla^k$ and $\Delta$, which can be proved using the Ricci identity for commuting the covariant derivatives of tensor, i.e., for a $k$-tensor $A$ on $\M$:
\begin{equation*}
  (\nabla_i\nabla_j-\nabla_j\nabla_i)A_{i_1i_2\cdots i_k}=\sum_{l=1}^kR_{iji_l}^{\quad m}A_{i_1\cdots i_{l-1}mi_{l+1}\cdots i_k}.
\end{equation*}

\begin{lem}[{$\!\!$\cite{Bando}; \cite[Lemma 13.24]{Chow-etal-2008}}]\label{lem-comu-spa}
For any $p$-tensor $A$ with $p\geq 1$ and any integer $k\in \mathbb{N}$, we have
\begin{align*}
  \nabla^k\Delta A-\Delta\nabla^kA~\lessapprox & ~14(p+1)\sum_{i=0}^k\binom{k+2}{i+2}\nabla^iRm*\nabla^{k-i}A.
\end{align*}
\end{lem}

We also have the following commutator formula for $\nabla^k$ and $\frac{\pt}{\pt t}$ acting on a tensor along the Laplacian flow.
\begin{lem}\label{lem-comu-pt}
If $\varphi(t), t \in [0,T_0]$ is a solution to the Laplacian flow \eqref{Lap-flow-def} on a compact  manifold $\M$, then for any $p$-tensor $A$ with $p\geq 1$ and any integer $k\in \mathbb{N}$, we have
\begin{align}\label{lem-comu-pt-eqn}
  \nabla^k\frac{\pt}{\pt t} A-\frac{\pt}{\pt t}\nabla^kA~&\lessapprox ~ 21(p+1)\sum_{i=1}^k\binom{k+1}{i+1}   \nabla^iRm*\nabla^{k-i}A\nonumber\\
  & +13(p+1)\sum_{i=1}^k \binom{k+1}{i+1} 
  \nabla^{i}(T*T)*\nabla^{k-i}A.
\end{align}
\end{lem}
\proof
First, by a trivial adjustment to the proof of \cite[Lemma 13.26]{Chow-etal-2008}, for any smooth one-parameter family of metrics $g(t)$ on $\M$ evolving by \eqref{evl-g} for any smooth family of symmetric $2$-tensors $h(t)$,
we have
\begin{align}\label{lem-comu-pf0}
  \nabla^k\frac{\pt}{\pt t} A-\frac{\pt}{\pt t}\nabla^kA~\lessapprox & ~ 3(p+1)\sum_{i=1}^k   \binom{k+1}{i+1}\nabla^i h*\nabla^{k-i}A.
\end{align}
Under the Laplacian flow, $g(t)$ evolves by \eqref{flow-g2}, so
\begin{equation}\label{lem-comu-pf2}
  \nabla^i h(t)\lessapprox ~7\nabla^iRm+\frac {13}3\nabla^{i}(T*T).
\end{equation}
The commutator formula \eqref{lem-comu-pt-eqn} follows by substituting \eqref{lem-comu-pf2} into \eqref{lem-comu-pf0}.
\endproof

Combining Lemmas \ref{lem-comu-spa} and \ref{lem-comu-pt}, we have the following commutator formula of $\nabla^k$ and the heat operator $\frac{\pt}{\pt t}-\Delta$ acting on a tensor.
\begin{prop}
If $\varphi(t), t \in [0,T_0]$ is a solution to the Laplacian flow \eqref{Lap-flow-def} 
on a compact  manifold $\M$, then  for any $p$-tensor $A$ with $p\geq 1$ and any integer $k\in \mathbb{N}$, we have
\begin{align}\label{prop-commu-heat}
  \left(\frac{\pt}{\pt t}-\Delta\right)\nabla^kA &-\nabla^k\left(\frac{\pt}{\pt t}-\Delta\right) A\nonumber\\
  \lessapprox & ~ 21(p+1)\sum_{i=0}^k\frac{i+k+4}{i+2}\binom{k+1}{i+1}\nabla^iRm*\nabla^{k-i}A\nonumber\\
  & +13(p+1)\sum_{i=1}^k\binom{k+1}{i+1}\nabla^{i}(T*T)*\nabla^{k-i}A.
\end{align}
\end{prop}

\subsection{Main derivative estimate}
Our main estimate is the following, recalling the quantity $\Lambda(x,t)$ given
in \eqref{lambda-eq}.
\begin{thm}\label{thm-deri-est}
Suppose that $\varphi(t), t \in [0,T_0]$ is a solution to the Laplacian flow \eqref{Lap-flow-def} 
on a compact  manifold $\M$. There exists a universal positive constant $\alpha$ and a positive constant $C_*=C_*(T_0,K_0)$, where $K_0=\sup_\M|\Lambda(x,0)|$, such that
\begin{equation}\label{deri-est}
  \sum_{k=0}^N\frac {t^k}{(k+1)!^2}\left(|\nabla^kRm|^2(x,t)+|\nabla^{k+1}T|^2(x,t)+|\nabla^{k+2}\varphi|^2(x,t)\right)\leq C_*
\end{equation}
on $\M\times [0, \min\{T_0,\alpha/K_0\}]$ for all $N\in \mathbb{N}$.
\end{thm}
For convenience, we define
\begin{align}
 a_k&=\frac{t^{\frac{k}{2}}|\nabla^kRm|}{(k+1)!},& b_k&=\frac{t^{\frac{k}{2}}|\nabla^{k+1}T|}{(k+1)!}, & c_k&=\frac{t^{\frac{k}{2}}|\nabla^{k+2}\varphi|}{(k+1)!},&\textrm{for }k\geq 0,\label{def-ak}\\
  \tilde{a}_k&=\frac{t^{\frac{k-1}2}|\nabla^kRm|}{k!}, & \tilde{b}_k&=\frac{t^{\frac{k-1}2}|\nabla^{k+1}T|}{k!},& \tilde{c}_k&=\frac{t^{\frac{k-1}2}|\nabla^{k+2}\varphi|}{k!}, &\textrm{for }k\geq 1.\label{def-ak-tilde}
\end{align}
By setting $k!=1$ for all $k\leq 0$, the above definition can cover
\begin{align*}
  \tilde{a}_0&=t^{-\frac 12}|Rm|, & \tilde{b}_0&=t^{-\frac 12}|\nabla T|,& \tilde{c}_0&=t^{-\frac 12}|\nabla^2\varphi|,\\
  b_{-1}&=t^{-\frac 12}|T|,& c_{-1}&=t^{-\frac 12}|\nabla\varphi|, & c_{-2}&=t^{-1}|\varphi|,\\
  \tilde{b}_{-1}&=t^{-1}|T|, & \tilde{c}_{-1}&=t^{-1}|\nabla\varphi|, & \tilde{c}_{-2}&=t^{-\frac 32}|\varphi|.
\end{align*}
Note that $|\varphi|^2=7$ and $|\nabla\varphi|^2=|T|^2\leq |Rm|=a_0$. Next, we define
\begin{align*}
  A_N&=\sum_{k=0}^Na_k^2, & B_N&=\sum_{k=0}^Nb_k^2, & C_N&=\sum_{k=0}^Nc_k^2,\\
  \tilde{A}_N&=\sum_{k=1}^N\tilde{a}_k^2,   & \tilde{B}_N&=\sum_{k=1}^N\tilde{b}_k^2, &  \tilde{C}_N&=\sum_{k=1}^N\tilde{c}_k^2,
\end{align*}
and
\begin{align*}
  \Phi_N= A_N+B_N+C_N,\quad \quad \Psi_N= \tilde{A}_N+\tilde{B}_N+\tilde{C}_N.
\end{align*}
Then \eqref{deri-est} is equivalent to showing that $\Phi_N\leq C_*$ for any  $N\in \mathbb{N}$.

The approach to prove \eqref{deri-est} is to establish an evolution inequality for $\Phi_N$ and then apply the maximum
 principle.  Although the method is clear, the derivation of the evolution inequality is somewhat computationally involved, so we break it up into a
  sequence of lemmas which deals with each of  the terms $A_N$, $B_N$ and $C_N$ in turn.  Throughout the proofs we will use the same symbol $C$ to denote a
   (finite) universal constant.

\begin{lem}\label{lem-4-5}
Suppose that $\varphi(t), t \in [0,T_0]$ is a solution to the Laplacian flow \eqref{Lap-flow-def} 
on a compact manifold $\M$. There exists a universal constant $C$  such that
\begin{align}\label{lem-4-5-est}
  \left(\frac{\pt}{\pt t}-\Delta\right)A_N \leq &-\frac 74\tilde{A}_{N+1}+\frac 14\Psi_{N+1}+C(t\Phi_N^{\frac 12}+t^2\Phi_N)\Psi_N\nonumber\\
  &+C\Phi_N^{\frac 32}(1+t\Phi_N^{\frac 12}).
\end{align}
\end{lem}
\proof

Applying \eqref{prop-commu-heat} to $A=Rm$ (where $p=4$), we have
\begin{align}\label{lem-4-5-pf0}
  \left(\frac{\pt}{\pt t}-\Delta\right)\nabla^kRm
  \lessapprox & ~\nabla^k \left(\frac{\pt}{\pt t}-\Delta\right)Rm\nonumber\\
  &~ +105\sum_{i=0}^k\frac{i+k+4}{i+2}\binom{k+1}{i+1}\nabla^iRm*\nabla^{k-i}Rm\nonumber\\
  &~ +65\sum_{i=1}^k\binom{k+1}{i+1}\nabla^{i}(T*T)*\nabla^{k-i}Rm.
\end{align}
Applying $\nabla^k$ to \eqref{evl-Rm-0} and substituting into \eqref{lem-4-5-pf0}, we obtain
 \begin{align*}
 \left(\frac{\pt}{\pt t}-\Delta\right)\nabla^kRm &\lessapprox ~   243
 \sum_{i=0}^k\binom{k+2}{i+2}\nabla^iRm*\nabla^{k-i}Rm\nonumber\\
   & ~ +35\nabla^{k+1}(\nabla T*T)+4T*T*\nabla^kRm\nonumber\\
   &~+69\sum_{i=1}^k\binom{k+1}{i+1}\nabla^{i}(T*T)*\nabla^{k-i}Rm.
\end{align*}
Since
\begin{equation*}
  |\nabla^kRm|^2=(g^{-1})^{*(k+4)}*\nabla^kRm*\nabla^kRm,
\end{equation*} 
we can use the evolution equation \eqref{flow-g^-1} of $g^{-1}$ to compute
\begin{align*}
  &\bigg(\frac{\pt}{\pt t}-\Delta\bigg)|\nabla^kRm|^2\\ &=2\bigg\langle\left(\frac{\pt}{\pt t}-\Delta\right)\nabla^kRm,\nabla^kRm\bigg\rangle-2|\nabla^{k+1}Rm|^2 \\
   &+ 2(k+4)\left(Rc+\frac 13|T|^2g+2T*T\right)*\nabla^kRm*\nabla^kRm\displaybreak[0]\\
   &\leq -2|\nabla^{k+1}Rm|^2+C\sum_{i=0}^k\binom{k+2}{i+2}|\nabla^iRm||\nabla^{k-i}Rm||\nabla^kRm|\nonumber\\
   &+C|\nabla^{k+1}(\nabla T*T)||\nabla^kRm|   +C\sum_{i=0}^k\binom{k+1}{i+1}|\nabla^{i}(T*T)||\nabla^{k-i}Rm||\nabla^kRm|.
\end{align*}

Then from the definition \eqref{def-ak} of $a_k$, we have
\begin{align}\label{lem-4-5-pf2}
  \left(\frac{\pt}{\pt t}-\Delta\right)a_k^2\leq & -2\tilde{a}_{k+1}^2+\frac {k\tilde{a}_k^2}{(k+1)^2}+I_{1}(k)+I_2(k)+I_3(k),
\end{align}
where
\begin{align*}
  I_1(k)=&~\frac {Ct^{ k}}{(k+1)!^2}\sum_{i=0}^k \binom{k+2}{i+2}|\nabla^iRm||\nabla^{k-i}Rm||\nabla^kRm|\\
   I_2(k)=&~\frac {Ct^{ k}}{(k+1)!^2}|\nabla^{k+1}(\nabla T*T)||\nabla^kRm|\\
   I_3(k)=&~\frac {Ct^{ k}}{(k+1)!^2}\sum_{i=0}^k\binom{k+1}{i+1}|\nabla^{i}(T*T)||\nabla^{k-i}Rm||\nabla^kRm|.
 \end{align*}
To obtain \eqref{lem-4-5-est} we sum \eqref{lem-4-5-pf2} from $k=0$ to $N$.  First, for $k=0$, we have
\begin{align}\label{lem-4-5-pf3}
   \left(\frac{\pt}{\pt t}-\Delta\right)a_0^2\leq  & -2\tilde{a}_1^2+C\left(a_0^3+b_0^2a_0+\tilde{b}_1a_0^{\frac 32}\right)\nonumber\\
   \leq &-2\tilde{a}_1^2+C\Phi_N^{\frac 32}+C_n\tilde{B}_N^{\frac 12}\Phi_N^{\frac 34}.
\end{align}
For $k=1$ to $N$, we estimate the sum over $k$ of the three terms $I_{1}(k)$, $I_2(k)$, $I_3(k)$ separately.  For $I_1(k)$ we have
\begin{align}\label{lem-4-5I1}
  \sum_{k=1}^NI_1(k) &=  \sum_{k=1}^N\left(Ca_0a_k^2 +Ct\sum_{i=0}^{k-1}\frac {(k+2)a_i\tilde{a}_{k-i}\tilde{a}_k}{(k+1)(i+2)}\right)\nonumber\\
 & \leq  CA_N^{\frac 32}+Ct\sum_{k=1}^N\left(\sum_{i=0}^{k-1}\frac 1{(i+2)^2}\sum_{i=0}^{k-1}a_i^2\tilde{a}_{k-i}^2\right)^{\frac 12}\tilde{a}_k\nonumber\displaybreak[0]\\
  &\leq CA_N^{\frac 32}+Ct\left(\sum_{k=1}^N\sum_{i=0}^{k-1}a_i^2\tilde{a}_{k-i}^2\right)^{\frac 12}\left(\sum_{k=1}^N\tilde{a}_k^2\right)^{\frac 12}\nonumber\displaybreak[0]\\
   &\leq CA_N^{\frac 32}+CtA_N^{\frac 12}\tilde{A}_N\nonumber\\ &\leq C\Phi_N^{\frac 32}+Ct\Phi_N^{\frac 12}\Psi_N,
\end{align}
where we used $a_0\leq A_N^{\frac 12}$, the Cauchy--Schwarz inequality and the elementary fact that
  $\sum_{i=0}^{\infty}\frac 1{(i+2)^2}<1$.
 For the sum of $I_2(k)$,
\begin{align}\label{lem-4-5I2}
  \sum_{k=1}^N I_2(k) &=  \sum_{k=1}^N\left(C\tilde{b}_{k+1}a_0^{\frac 12}a_k+Cb_0b_ka_k+Ct\sum_{i=1}^{k}\frac{\tilde{b}_ib_{k-i}\tilde{a}_k}{k+1}\right)\nonumber\displaybreak[0]\\
  &\leq C\tilde{B}_{N+1}^{\frac 12}A_N^{\frac 34}+CB_NA_N^{\frac 12}+Ct\sum_{k=1}^N\frac{k^{\frac 12}}{k+1}\left(\sum_{i=1}^{k}\tilde{b}_i^2b_{k-i}^2\right)^{\frac 12}\tilde{a}_k\nonumber\displaybreak[0]\\
 &\leq C\tilde{B}_{N+1}^{\frac 12}A_N^{\frac 34}+CB_NA_N^{\frac 12}+Ct\tilde{B}_N^{\frac 12}B_N^{\frac 12}\tilde{A}_N^{\frac 12}\nonumber\displaybreak[0]\\
  &\leq C\tilde{B}_{N+1}^{\frac 12}\Phi_N^{\frac 34}+C\Phi_N^{\frac 32}+Ct\Phi_N^{\frac 12}\Psi_N,
\end{align}
where we used the elementary inequality
\begin{equation*}
  \left(\sum_{i=1}^k\alpha_i\right)^2\leq k\sum_{i=1}^k\alpha_i^2,\quad \textrm{for }\alpha_i\geq 0.
\end{equation*}
We can similarly estimate the sum of $I_3(k)$:
\begin{align}\label{lem-4-5I3}
   &\sum_{k=1}^N I_3(k)  =C\sum_{k=1}^N\frac {t^{k}}{(k+1)!^2}|\nabla^k(T*T)||Rm||\nabla^kRm|\nonumber\\
  &\quad\qquad\qquad +Ct\sum_{k=1}^N\sum_{i=0}^{k-1}\frac {\tilde{a}_{k-i}}{k+1}\left(\frac {t^{\frac i2}|\nabla^i(T*T)|}{(i+1)!}\right)\tilde{a}_k\nonumber\\
 & \leq Cta_0\sum_{k=1}^N\sum_{i=0}^k\frac {a_k}{k+1}b_{i-1}b_{k-i-1}  +Ct^2\sum_{k=1}^N\sum_{i=0}^{k-1}\frac {\tilde{a}_{k-i}}{k+1}\left(\sum_{j=0}^i\frac {b_{j-1}b_{i-j-1}}{i+1}\right)\tilde{a}_k\nonumber\\
  &\leq Cta_0(B_N+b_{-1}^2)A_N^{\frac 12}+Ct^2(B_N+b_{-1}^2)\tilde{A}_N\nonumber\\
  &\leq Ct\Phi_N^2+C\Phi_N^{\frac 32}+Ct^2\Phi_N\Psi_N+Ct\Phi_N^{\frac 12}\Psi_N,
\end{align}
where in the third inequality we used that $a_0\leq A_N^{\frac 12}$ and $b_{-1}^2\leq t^{-1}A_N^{\frac 12}$.
Combining \eqref{lem-4-5-pf2}--\eqref{lem-4-5I3}, we conclude that
\begin{align*}
  \left(\frac{\pt}{\pt t}-\Delta\right)A_N \leq & -2\tilde{A}_{N+1}+\sum_{k=1}^N\frac {k\tilde{a}_k^2}{(k+1)^2}+C\tilde{B}_{N+1}^{\frac 12}\Phi_N^{\frac 34}+C\tilde{B}_N^{\frac 12}\Phi_N^{\frac 34}\\
  &+C\Phi_N^{\frac 32}(1+t\Phi_N^{\frac 12})+C(t\Phi_N^{\frac 12}+t^2\Phi_N)\Psi_N\nonumber\\
  \leq &-\frac 74\tilde{A}_{N+1}+\frac 14\Psi_{N+1}+C(t\Phi_N^{\frac 12}+t^2\Phi_N)\Psi_N\nonumber\\
  & +C\Phi_N^{\frac 32}(1+t\Phi_N^{\frac 12}),
\end{align*}
where we used $\frac k{(k+1)^2}\leq 1/4$, $\tilde{B}_{N}\leq \tilde{B}_{N+1}\leq \Psi_{N+1}$ and Cauchy--Schwarz.
\endproof

\begin{lem}\label{lem-4-6}
Suppose that $\varphi(t), t \in [0,T_0]$ is a solution to the Laplacian flow \eqref{Lap-flow-def} 
on a compact  manifold $\M$. There exists a universal constant $C$ such that
\begin{align}\label{evl-tors-2}
\left(\frac{\pt}{\pt t}-\Delta\right)B_N \leq & -\frac 74\tilde{B}_{N+1}+\frac 14\Psi_{N+1}+C\biggl(t\Phi_N^{\frac 12}+t^2\Phi_N+t^3\Phi_N^{\frac 32}\biggr)\Psi_N\nonumber\\
&+C\Phi_N^{\frac 32}\left(1+t\Phi_N^{\frac 12}\right).
\end{align}
\end{lem}
\proof
 By \eqref{evl-tors-0} and \eqref{prop-commu-heat} (with $A=T$ so $p=2$), we have
\begin{align}\label{lem-4-6-pf2}
   &\left(\frac{\pt}{\pt t}-\Delta\right)\nabla^{k+1}T \lessapprox
   150\sum_{i=0}^{k+1}\binom{k+3}{i+2}\nabla^iRm*\nabla^{k+1-i}T\nonumber\\
  &+\sum_{i=0}^{k+1}\binom{k+1}{i}\nabla^iRm*\nabla^{k+2-i}\varphi+
  11\sum_{i=0}^{k+1}\binom{k+1}{i}\nabla^{i}(\nabla T*T)*\nabla^{k+1-i}\varphi\nonumber\\
   &+43\sum_{i=1}^{k+1}\binom{k+2}{i+1}\nabla^{i}(T*T)*\nabla^{k+1-i} T+4T*T*\nabla^{k+1}T.
\end{align}
From the definition \eqref{def-ak} of $b_k$ and using $|T|^2\leq |Rm|$ we have
\begin{align}\label{lem-4-6-pf3}
  \bigg(\frac{\pt}{\pt t}&-\Delta\bigg)b_k^2 =\frac{t^k}{(k+1)!^2}\left(\frac{\pt}{\pt t}-\Delta\right)|\nabla^{k+1}T|^2+\frac k{(k+1)^2}\tilde{b}_k^2 \nonumber\\
  &=   \frac{2t^k}{(k+1)!^2}\bigg\langle\left(\frac{\pt}{\pt t}-\Delta\right)\nabla^{k+1}T,\nabla^{k+1}T\bigg\rangle-2\tilde{b}_{k+1}^2+\frac k{(k+1)^2}\tilde{b}_k^2\nonumber\\
  &\quad+\frac{2t^k(k+3)}{(k+1)!^2}\left(Rc+\frac 13|T|^2g+2T*T\right)*\nabla^{k+1}T*\nabla^{k+1}T\nonumber\\
 & \leq  \frac{2t^k}{(k+1)!^2}\bigg\langle\left(\frac{\pt}{\pt t}-\Delta\right)\nabla^{k+1}T,\nabla^{k+1}T\bigg\rangle-2\tilde{b}_{k+1}^2\nonumber\\  &\quad
  +\frac k{(k+1)^2}\tilde{b}_k^2+C(k+3)|Rm|b_k^2.
\end{align}
 Substituting \eqref{lem-4-6-pf2} into \eqref{lem-4-6-pf3} and rearranging terms gives:
\begin{align}\label{lem-4-6-pf4}
  \left(\frac{\pt}{\pt t}-\Delta\right)&b_k^2 \leq -2\tilde{b}_{k+1}^2+\frac {k\tilde{b}_k^2}{(k+1)^2}+II_1(k)+\cdots+II_4(k),
\end{align}
where 
\begin{align*}
 II_1(k) =&~\frac{Ct^k}{(k+1)!^2}\sum_{i=0}^{k+1}\binom{k+3}{i+2}|\nabla^iRm||\nabla^{k+1-i}T||\nabla^{k+1}T|\nonumber\\
  II_2(k)=&~\frac{t^k}{(k+1)!^2}\sum_{i=0}^{k+1}\binom{k+1}{i}|\nabla^iRm||\nabla^{k+2-i}\varphi||\nabla^{k+1}T|\nonumber\\
 II_3(k) =&~\frac{Ct^k}{(k+1)!^2}\sum_{i=0}^{k+1}\binom{k+1}{i}|\nabla^{i}(\nabla T*T)||\nabla^{k+1-i}\varphi||\nabla^{k+1}T|\nonumber\\
  II_4(k)= &~\frac{Ct^k}{(k+1)!^2}\sum_{i=1}^{k+1}\binom{k+2}{i+1}|\nabla^{i}(T*T)||\nabla^{k+1-i} T||\nabla^{k+1}T|.
\end{align*}
Note that we have absorbed $C(k+3)|Rm|b_k^2$ in \eqref{lem-4-6-pf3} into $II_1(k)$ of \eqref{lem-4-6-pf4}. To derive the evolution inequality of $B_N$, we start with \eqref{lem-4-6-pf4} for $k=0$:
\begin{align}\label{lem-4-6-pf5}
  \left(\frac{\pt}{\pt t}-\Delta\right)b_0^2\leq &-2\tilde{b}_1^2+C\left(a_0b_0^2+\tilde{a}_1a_0^{\frac 12}b_0+a_0b_0c_0 +\tilde{b}_1a_0^{\frac 12}b_0+b_0^3\right)\nonumber\displaybreak[0]\\
  \leq &-2\tilde{b}_1^2+C\Phi_N^{\frac 32}+C\Psi_N^{\frac 12}\Phi_N^{\frac 34}.
\end{align}
By summing $II_1(k),II_2(k),II_3(k)$ over $k=1,\cdots,N$, we have the following estimates:
\begin{align}\label{lem-4-6-II1}
  \sum_{k=1}^NII_1(k) &=  C\sum_{k=1}^N\left(\tilde{a}_{k+1}a_0^{\frac 12}b_k+t^{\frac 12}b_0a_k\tilde{b}_k\right)\nonumber\\
  &+Ct\sum_{k=1}^N\sum_{i=0}^{k-1}\left(\frac 1{i+2}+\frac 1{k+1-i}\right)a_i\tilde{b}_{k-i}\tilde{b}_k\nonumber\\
 & \leq C\tilde{A}_{N+1}^{\frac 12}A_N^{\frac 14}B_N^{\frac 12}+Ct^{\frac 12}B_N^{\frac 12}A_N^{\frac 12}\tilde{B}_N^{\frac 12}+CtA_N^{\frac 12}\tilde{B}_N\nonumber\displaybreak[0]\\
  &\leq C\Psi_{N+1}^{\frac 12}\Phi_N^{\frac 34}+Ct^{\frac 12}\Phi_N\Psi_N^{\frac 12}+Ct\Phi_N^{\frac 12}\Psi_N;\\
\label{lem-4-6-II2}
  \sum_{k=1}^NII_2(k) &=  C\sum_{k=1}^N\left(a_0b_kc_k+a_0^{\frac 12}\tilde{a}_{k+1}b_k\right)+Ct\sum_{k=1}^N\sum_{i=1}^{k}\frac{\tilde{a}_ic_{k-i}\tilde{b}_k}{k+1}\nonumber\\
 & \leq CA_N^{\frac 12}B_N^{\frac 12}C_N^{\frac 12}+CA_N^{\frac 14}\tilde{A}_{N+1}^{\frac 12}B_N^{\frac 12}+Ct\tilde{A}_N^{\frac 12}C_N^{\frac 12}\tilde{B}_N^{\frac 12}\nonumber\\
  &\leq C\Phi_N^{\frac 32}+C\Psi_{N+1}^{\frac 12}\Phi_N^{\frac 34}+Ct\Phi_N^{\frac 12}\Psi_N;\displaybreak[0]\\
\label{lem-4-6-II3}
\sum_{k=1}^NII_3(k)& \leq Ct\sum_{k=1}^N\frac {\tilde{b}_k}{k+1}\sum_{i=0}^{k+1}\tilde{b}_ib_{k-i}\nonumber\\
&\quad+Ct^2\sum_{k=1}^N\frac {\tilde{b}_k}{k+1}\sum_{i=0}^{k}\left(\sum_{j=0}^i\frac {\tilde{b}_jb_{i-j-1}}{k+1-i}\right)c_{k-i-1}\nonumber\\
&\leq Ct(\tilde{B}_{N+1}+\tilde{b}_0^2)^{\frac 12}(B_N+b_{-1}^2)^{\frac 12}\tilde{B}_N^{\frac 12}\nonumber\\
 & \quad+Ct^2\sum_{k=1}^N\left(\sum_{i=0}^{k}\sum_{j=0}^i\tilde{b}^2_jb^2_{i-j-1}c^2_{k-i-1}\right)^{\frac 12}\tilde{b}_k\nonumber\\
 &\leq Ct(\tilde{B}_{N+1}+\tilde{b}_0^2)^{\frac 12}(B_N+b_{-1}^2)^{\frac 12}\tilde{B}_N^{\frac 12}\nonumber\\
 &\quad+C_nt^2(\tilde{B}_{N}+\tilde{b}_0^2)^{\frac 12}(B_N+b_{-1}^2)^{\frac 12}(C_N+c_{-1}^2)^{\frac 12}\tilde{B}_N^{\frac 12}\nonumber\\
&\leq C_n\Psi_{N+1}^{\frac 12}\Psi_N^{\frac 12}t^{\frac 12}\Phi_N^{\frac 14}(1+t\Phi_N^{\frac 12})^{\frac 12}+C_n\Phi_N^{\frac 34}(1+t\Phi_N^{\frac 12})^{\frac 12}\Psi_N^{\frac 12}\nonumber\\
&\quad+C_n(t^2\Phi_N+t\Phi_N^{\frac 12})\Psi_N+C_nt^{\frac 12}\Phi_N(1+t\Phi_N^{\frac 12})\Psi_N^{\frac 12},
\end{align}
where we used $b_{-1}^2\leq t^{-1}a_0$ and $c_{-1}^2\leq t^{-1}a_0$. Finally,
\begin{align}\label{lem-4-6-II4}
 & \sum_{k=1}^NII_4(k)=  C\sum_{k=1}^N\sum_{i=1}^{k+1}\frac{(k+2)b_{k-i}\tilde{b}_k}{(k+1)(i+1)}\left(ta_0^{\frac 12}b_{i-1}+t^2\sum_{j=1}^i\frac 1j\tilde{b}_{j-1}b_{i-j-1}\right)\nonumber\displaybreak[0]\\
  &\leq Cta_0^{\frac 12}B_N^{\frac 12}(B_N+b_{-1}^2)^{\frac 12}\tilde{B}_N^{\frac 12}+Ct^2\sum_{k=1}^N\left(\sum_{i=1}^{k+1}\sum_{j=1}^i\tilde{b}_{j-1}^2b_{i-j-1}^2b_{k-i}^2\right)^{\frac 12}\tilde{b}_k\nonumber\displaybreak[0]\\
  &\leq Cta_0^{\frac 12}B_N^{\frac 12}(B_N+b_{-1}^2)^{\frac 12}\tilde{B}_N^{\frac 12}+Ct^2\left(\tilde{B}_N+\tilde{b}_0^2\right)^{\frac 12}(B_N+b_{-1}^2)\tilde{B}_N^{\frac 12}\nonumber\displaybreak[0]\\
  &\leq Ct^{\frac 12}\Phi_N(1+t\Phi_N^{\frac 12})^{\frac 12}\Psi_N^{\frac 12}+Ct^{\frac 12}\Phi_N(1+t\Phi_N^{\frac 12})\Psi_N^{\frac 12}+C(t^2\Phi_N+t\Phi_N^{\frac 12})\Psi_N.
\end{align}
Combining the above inequalities \eqref{lem-4-6-pf5}--\eqref{lem-4-6-II4}, we have
\begin{align*}
\left(\frac{\pt}{\pt t}-\Delta\right)B_N &\leq -\frac 74\tilde{B}_{N+1}+C\Phi_N^{\frac 32}+C\Psi_N^{\frac 12}\Phi_N^{\frac 34}+C\Psi_{N+1}^{\frac 12}\Phi_N^{\frac 34}\nonumber\\
&+C\Psi_{N+1}^{\frac 12}\Psi_N^{\frac 12}t^{\frac 12}\Phi_N^{\frac 14}(1+t\Phi_N^{\frac 12})^{\frac 12}\nonumber\displaybreak[0]\\
&+C\Phi_N^{\frac 34}(1+t\Phi_N^{\frac 12})^{\frac 12}\Psi_N^{\frac 12}+C(t^2\Phi_N+t\Phi_N^{\frac 12})\Psi_N\nonumber\displaybreak[0]\\
&+Ct^{\frac 12}\Phi_N(1+t\Phi_N^{\frac 12})\Psi_N^{\frac 12}+Ct^{\frac 12}\Phi_N(1+t\Phi_N^{\frac 12})^{\frac 12}\Psi_N^{\frac 12}.
\end{align*}
Noting that $\Psi_N\leq \Psi_{N+1}$ and applying Cauchy--Schwarz to the above inequality gives \eqref{evl-tors-2}.
\endproof

\begin{lem}\label{lem-4-7}
Suppose that $\varphi(t), t \in [0,T_0]$ is a solution to the Laplacian flow \eqref{Lap-flow-def} on a compact manifold $\M$. There exists a universal constant $C$ 
 such that
\begin{align}\label{lem-4-7-est}
  \left(\frac{\pt}{\pt t}-\Delta\right)C_N\leq& -\frac 74\tilde{C}_{N+1}+\frac 14\Psi_{N+1}+C\Phi_N^{\frac 32}\left(1+t\Phi_N^{\frac 12}+t^2\Phi_N\right)\nonumber\\
  &+C\Psi_N\left(t\Phi_N^{\frac 12}+t^2\Phi_N+t^3\Phi_N^{\frac 32}+t^4\Phi_N^2\right).
\end{align}
\end{lem}
\proof
By \eqref{evl-nab-var-0} and \eqref{prop-commu-heat} (with $A=\nabla\varphi$ so $p=4$), we have
\begin{align}\label{lem-4-7-pf1}
&\left(\frac{\pt}{\pt t}-\Delta\right)\nabla^{k+2}\varphi\lessapprox  62\sum_{i=0}^{k+1}\binom{k+1}{i}\nabla^{i}( \nabla T*T)*\nabla^{k+1-i}\varphi\nonumber\\
&+
239\sum_{i=0}^{k+1}\binom{k+3}{i+2}\nabla^iRm*\nabla^{k+2-i}\varphi+\sum_{i=0}^{k+1}\binom{k+1}{i}\nabla^iRm*\nabla^{k+1-i}T\nonumber\\
&+\sum_{i=0}^{k+1}\binom{k+1}{i}\nabla^{i}(Rm*\varphi)*\biggl(\nabla^{k+1-i}(T*\varphi)+\nabla^{k+1-i}(\nabla\varphi*\varphi)\biggr)\nonumber\\
   &+ 89\sum_{i=1}^{k+1}\binom{k+2}{i+1}\nabla^{i}(T*T)*\nabla^{k+2-i}\varphi+24T*T*\nabla^{k+2}\varphi\nonumber\\
   &+6\sum_{i=0}^{k+1}\binom{k+1}{i}\nabla^{i}(\nabla T*\nabla\varphi)*\nabla^{k+1-i}\varphi.
\end{align}
By the the definition \eqref{def-ak} of $c_k$ and noting that $\nabla^{k+2}\varphi$ is an $(k+5)$-tensor, we have the following:
\begin{align}\label{lem-4-7-pf2}
  \left(\frac{\pt}{\pt t}-\Delta\right)&c_k^2 = \frac{t^k}{(k+1)!^2}\left(\frac{\pt}{\pt t}-\Delta\right)|\nabla^{k+2}\varphi|^2+\frac k{(k+1)^2}\tilde{c}_k^2 \nonumber\\
  &\leq  \frac{2t^k}{(k+1)!^2}\bigg\langle\left(\frac{\pt}{\pt t}-\Delta\right)\nabla^{k+2}\varphi,\nabla^{k+2}\varphi\bigg\rangle-2\tilde{c}_{k+1}^2 \nonumber\\&
  \quad +\frac k{(k+1)^2}\tilde{c}_k^2+C(k+5)|Rm|c_k^2.
\end{align}
Substituting \eqref{lem-4-7-pf1} into \eqref{lem-4-7-pf2}, we compute
\begin{align}\label{lem-4-7-pf3}
  &\left(\frac{\pt}{\pt t}-\Delta\right)c_k^2\leq  -2\tilde{c}_{k+1}^2+\frac k{(k+1)^2}\tilde{c}_k^2+III_1(k)+\cdots+III_6(k),
  \end{align}
  where
\begin{align*}
 III_1(k) &=C\frac{t^k}{(k+1)!^2}\sum_{i=1}^{k+1}\binom{k+1}{i}|\nabla^{i}( \nabla T*T)||\nabla^{k+1-i}\varphi||\nabla^{k+2}\varphi|,\displaybreak[0]\\
III_2(k)&=C\frac{t^k}{(k+1)!^2}\sum_{i=0}^{k+1}\binom{k+3}{i+2}|\nabla^iRm||\nabla^{k+2-i}\varphi||\nabla^{k+2}\varphi|,\displaybreak[0]\\
III_3(k)&=\frac{2t^k}{(k+1)!^2}\sum_{i=0}^{k+1}\binom{k+1}{i}|\nabla^iRm||\nabla^{k+1-i}T||\nabla^{k+2}\varphi|,\displaybreak[0]\\
III_4(k)&=\frac{2t^k}{(k+1)!^2}\sum_{i=0}^{k+1}\binom{k+1}{i}|\nabla^{i}(Rm*\varphi)|\biggl(|\nabla^{k+1-i}(T*\varphi)|\\
&\qquad\qquad\qquad\qquad\qquad\qquad\qquad+|\nabla^{k+1-i}(\nabla\varphi*\varphi)|\biggr)|\nabla^{k+2}\varphi|,\displaybreak[0]\\
  III_5(k)&=C\frac{t^k}{(k+1)!^2}\sum_{i=1}^{k+1}\binom{k+2}{i+1}|\nabla^{i}(T*T)||\nabla^{k+2-i}\varphi||\nabla^{k+2}\varphi|,\\
III_6(k)&=\frac{12t^k}{(k+1)!^2}\sum_{i=0}^{k+1}\binom{k+1}{i}|\nabla^{i}(\nabla T*\nabla\varphi)||\nabla^{k+1-i}\varphi||\nabla^{k+2}\varphi|.
\end{align*}
We now follow similar calculations to the proofs of Lemmas \ref{lem-4-5} and \ref{lem-4-6}.  For $k=0$ in \eqref{lem-4-7-pf3}:
\begin{align}\label{lem-4-7-pf4}
  \left(\frac{\pt}{\pt t}-\Delta\right)c_0^2\leq & -2\tilde{c}_1^2+C\biggl(\tilde{b}_1a_0^{\frac 12}c_0+b_0^2c_0 +b_0c_0^2+ a_0c_0^2\nonumber\\
   &\qquad+\tilde{a}_1a_0^{\frac 12}c_0+a_0b_0c_0+a_0^2c_0\biggr)\nonumber\\
   \leq &-2\tilde{c}_1^2+C\Phi_N^{\frac 32}+C\Psi_N^{\frac 12}\Phi_N^{\frac 34}.
\end{align}
We next estimate the sums of each of the six terms $III_1(k), \cdots, III_6(k)$ in \eqref{lem-4-7-pf3}.  Starting with $III_1(k)$ and using  $\tilde{b}_0=t^{-\frac 12}b_0$ and $b_{-1}^2=c_{-1}^2\leq t^{-1}a_0$:
\begin{align}\label{lem-4-7-III1}
  \sum_{k=1}^N III_1(k)& \leq C\sum_{k=1}^N\frac{t^k}{(k+1)!^2}|\nabla^{k+1}( \nabla T*T)||\varphi||\nabla^{k+2}\varphi|\nonumber\\
  &+C\sum_{k=1}^N\frac{t^k}{(k+1)!^2}\sum_{i=1}^{k} \binom{k+1}{i}|\nabla^{i}( \nabla T*T)||\nabla^{k+1-i}\varphi||\nabla^{k+2}\varphi|\nonumber\displaybreak[0]\\
  &\leq Ct\sum_{k=1}^N\frac 1{k+1}\sum_{i=0}^{k+1}\tilde{b}_ib_{k-i}\tilde{c}_k\nonumber\\
   &+Ct^2\sum_{k=1}^N\frac 1{k+1}\sum_{i=1}^{k}\frac 1{k+1-i}\sum_{j=0}^i\tilde{b}_jb_{i-j-1}c_{k-i-1}\tilde{c}_k \nonumber\displaybreak[0]\\
  & \leq Ct(\tilde{B}_{N+1}+\tilde{b}_0^2)^{\frac 12}(B_N+b_{-1}^2)^{\frac 12}\tilde{C}_N^{\frac 12} \nonumber\\
   &\quad +Ct^2(\tilde{B}_{N}+\tilde{b}_0^2)^{\frac 12}(B_N+b_{-1}^2)^{\frac 12}(C_N+c_{-1}^2)^{\frac 12}\tilde{C}_N^{\frac 12}\nonumber\\
  &\leq  C\Psi_{N+1}^{\frac 12}\Psi_N^{\frac 12}(t^2\Phi_N+t\Phi_N^{\frac 12})^{\frac 12}+C\Phi_N^{\frac 34}(1+t\Phi_N^{\frac 12})^{\frac 12}
  \Psi_N^{\frac 12}\nonumber\\
  &\quad+C(t^2\Phi_N+t\Phi_N^{\frac 12})^2\Psi_N+Ct^{\frac 12}\Phi_N(1+t\Phi_N^{\frac 12})\Psi_N^{\frac 12}.
\end{align}
 Using $|\nabla\varphi|\leq a_0^{\frac 12}\leq A_N^{\frac 14}$ for $III_2(k)$:
 
\begin{align}
  \sum_{k=1}^NIII_2(k) & \leq  Ct\sum_{k=1}^N\sum_{i=0}^{k-1}\frac {k+2}{k+1}\left(\frac 1{i+2}+\frac 1{k+1-i}\right)a_i\tilde{c}_{k-i}\tilde{c}_k \nonumber\displaybreak[0]\\
  & +C\sum_{k=1}^N\left(\tilde{a}_{k+1}|\nabla\varphi|c_k+\frac{k+3}{k+1}t^{\frac 12}\tilde{a}_kc_0c_k\right)\nonumber\displaybreak[0]\\
  &\leq  C tA_N^{\frac 12}\tilde{C}_N+C\tilde{A}_{N+1}^{\frac 12}A_N^{\frac 14}C_N^{\frac 12}+Ct^{\frac 12}\tilde{A}_N^{\frac 12}C_N\nonumber\\
  &\leq  C t\Phi_N^{\frac 12}\Psi_N+C\Psi_{N+1}^{\frac 12}\Phi_N^{\frac 34}+Ct^{\frac 12}\Phi_N\Psi_N^{\frac 12}.
\end{align}
Using $b_{-1}^2=c_{-1}^2\leq t^{-1}a_0$ again for $III_3(k)$ and $III_4(k)$:
\begin{align}
  \sum_{k=1}^NIII_3(k)&\leq  2t\sum_{k=1}^N\sum_{i=0}^{k-1}\frac 1{k+1}\tilde{a}_ib_{k-i}\tilde{c}_k +2\sum_{k=1}^N(\tilde{a}_{k+1}|T|c_k+t^{\frac 12}\tilde{a}_kb_0c_k)\nonumber\\
  &\leq  2t\tilde{A}_N^{\frac 12}B_N^{\frac 12}\tilde{C}_N^{\frac 12}+2\tilde{A}_{N+1}^{\frac 12}A_N^{\frac 14}C_N^{\frac 12}+2t^{\frac 12}\tilde{A}_N^{\frac 12}B_N^{\frac 12}C_N^{\frac 12}\nonumber\\
  &\leq  2t\Phi_N^{\frac 12}\Psi_N+2\Psi_{N+1}^{\frac 12}\Phi_N^{\frac 34}+2t^{\frac 12}\Phi_N\Psi_N^{\frac 12}.
\end{align}
\begin{align}
  \sum_{k=1}^N & III_4(k)
  \leq  2\sum_{k=1}^N\frac {t\tilde{c}_k}{k+1}\sum_{i=0}^{k+1}\left(\tilde{a}_i+t\sum_{j=0}^{i-1}\frac {\tilde{a}_jc_{i-j-2}}{i-j}\right)\nonumber\\&\qquad\qquad\qquad\qquad\qquad\qquad \times
  \left(b_{k-i}+t\sum_{l=0}^{k-i}\frac{(b_{l-1}+c_{l-1})c_{k-i-l-1}}{k+1-i-l}\right)\nonumber \displaybreak[0]\\
  &\leq 2t(\tilde{A}_{N+1}+\tilde{a}_0^2)^{\frac 12}\left(B_N+b_{-1}^2\right)^{\frac 12}\tilde{C}_N^{\frac 12}\nonumber\\
  &\quad+2t^2(\tilde{A}_{N}+\tilde{a}_0^2)^{\frac 12}(C_N+c_{-1}^2)^{\frac 12}(B_N+b_{-1}^2)^{\frac 12}\tilde{C}_N^{\frac 12}\nonumber\\
    &\quad+2t^2(\tilde{A}_{N+1}+\tilde{a}_0^2)^{\frac 12}(C_N+c_{-1}^2)^{\frac 12}(B_N+b_{-1}^2+C_N+c_{-1}^2)^{\frac 12}\tilde{C}_N^{\frac 12}\nonumber\\
  &\quad+2t^3(\tilde{A}_{N}+\tilde{a}_0^2)^{\frac 12}(C_N+c_{-1}^2)(B_N+b_{-1}^2+C_N+c_{-1}^2)^{\frac 12}\tilde{C}_N^{\frac 12}\nonumber\\
  &\leq 2\Psi_{N+1}^{\frac 12}\Psi_N^{\frac 12}(t^2\Phi_N+t\Phi_N^{\frac 12})^{\frac 12}+2\Phi_N^{\frac 34}(t\Phi_N^{\frac 12}+1)^{\frac 12}\Psi_N^{\frac 12}\nonumber\\
  &\quad+2\Psi_N(t^2\Phi_N+t\Phi_N^{\frac 12})+2\Phi_N^{\frac 34}(t\Phi_N^{\frac 12}+1)^{\frac 12}(t^2\Phi_N+t\Phi_N^{\frac 12})^{\frac 12}\Psi_N^{\frac 12}\nonumber\\
  &\quad+2\sqrt{2}\Psi_{N+1}^{\frac 12}\Psi_N^{\frac 12}(t^2\Phi_N+t\Phi_N^{\frac 12})+2\sqrt{2}t^{\frac 12}\Phi_N(t\Phi_N^{\frac 12}+1)\Psi_N^{\frac 12}\nonumber\\
  &\quad+2\Psi_N(t^2\Phi_N+t\Phi_N^{\frac 12})^{\frac 32}+2\Phi_N^{\frac 34}(t\Phi_N^{\frac 12}+1)^{\frac 12}(t^2\Phi_N+t\Phi_N^{\frac 12})\Psi_N^{\frac 12}.
  \end{align}
Finally, for $III_5(k)$ and $III_6(k)$ we have:
\begin{align}
  &\sum_{k=1}^NIII_5(k) \leq C\sum_{k=1}^N\sum_{i=1}^{k+1}\frac 1{i+1}\left(tb_{i-1}a_0^{\frac 12}+t^2\sum_{j=1}^i\frac 1j\tilde{b}_{j-1}b_{i-j-1}\right)c_{k-i}\tilde{c}_k \nonumber\\
 & \leq C\left(ta_0^{\frac 12}+t^2\left(\tilde{B}_N+\tilde{b}_0^2\right)^{\frac 12}\right)B_N^{\frac 12}(C_N+c_{-1}^2)^{\frac 12}\tilde{C}_N^{\frac 12}\nonumber\displaybreak[0]\\
  &\leq  C\Psi_N^{\frac 12}\Phi_N^{\frac 34}(t\Phi_N^{\frac 12}+1)+C\Psi_N^{\frac 12}t^{\frac 12}\Phi_N(t\Phi_N^{\frac 12}+1)+C(t^2\Phi_N+t\Phi_N^{\frac 12})\Psi_N;\displaybreak[0]\\
\label{lem-4-7-III6}
 & \sum_{k=1}^NIII_6(k)\leq  12\sqrt{7}t\sum_{k=1}^N\frac 1{k+1}\sum_{l=0}^{k+1}\tilde{b}_lc_{k-l}\tilde{c}_k\nonumber\\
  &\qquad\qquad\qquad+12t^2\sum_{k=1}^N\frac 1{k+1}\sum_{i=0}^{k}\frac {c_{k-i-1}\tilde{c}_k}{k+1-i}\left(\sum_{j=0}^i\tilde{b}_jc_{i-j-1}\right) \nonumber\\
 & \leq  Ct(\tilde{B}_{N+1}+\tilde{b}_0^2)^{\frac 12}(C_N+c_{-1}^2)^{\frac 12}\tilde{C}_N^{\frac 12}+Ct^2(\tilde{B}_{N}+\tilde{b}_0^2)^{\frac 12}(C_N+c_{-1}^2)\tilde{C}_N^{\frac 12}\nonumber\\
&\leq C\Psi_{N+1}^{\frac 12}\Psi_N^{\frac 12}(t^2\Phi_N+t\Phi_N^{\frac 12})^{\frac 12}+C(t^2\Phi_N+t\Phi_N^{\frac 12})\Psi_N\nonumber\\
  &\quad+C\Phi_N^{\frac 34}(t\Phi_N^{\frac 12}+1)^{\frac 12}\Psi_N^{\frac 12}+Ct^{\frac 12}\Phi_N(t\Phi_N^{\frac 12}+1)\Psi_N^{\frac 12}.
\end{align}
Combining \eqref{lem-4-7-pf3}--\eqref{lem-4-7-III6}, we obtain
\begin{align*}
  \left(\frac{\pt}{\pt t}-\Delta\right)&C_N\leq~ -\frac 74\tilde{C}_{N+1}+C\Phi_N^{\frac 32}+C\Psi_N^{\frac 12}\Phi_N^{\frac 34}+C\Psi_{N+1}^{\frac 12}\Phi_N^{\frac 34}\nonumber\\
  &+C\Psi_{N+1}^{\frac 12}\Psi_N^{\frac 12}(t^2\Phi_N+t\Phi_N^{\frac 12})^{\frac 12}+C\Phi_N^{\frac 34}(1+t\Phi_N^{\frac 12})^{\frac 12}\Psi_N^{\frac 12}\nonumber\displaybreak[0]\\
   &+C(t^2\Phi_N+t\Phi_N^{\frac 12})^2\Psi_N+Ct^{\frac 12}\Phi_N(1+t\Phi_N^{\frac 12})\Psi_N^{\frac 12}\nonumber\displaybreak[0]\\
   &+C\Psi_N(t^2\Phi_N+t\Phi_N^{\frac 12})+C\Psi_N(t^2\Phi_N+t\Phi_N^{\frac 12})^{\frac 32}\nonumber\displaybreak[0]\\
     &+2\Phi_N^{\frac 34}(t\Phi_N^{\frac 12}+1)^{\frac 12}(t^2\Phi_N+t\Phi_N^{\frac 12})^{\frac 12}\Psi_N^{\frac 12}\nonumber\\
  &+2\sqrt{2}\Psi_{N+1}^{\frac 12}\Psi_N^{\frac 12}(t^2\Phi_N+t\Phi_N^{\frac 12})+C\Psi_N^{\frac 12}\Phi_N^{\frac 34}(t\Phi_N^{\frac 12}+1)\nonumber\\
  &+2\Phi_N^{\frac 34}(t\Phi_N^{\frac 12}+1)^{\frac 12}(t^2\Phi_N+t\Phi_N^{\frac 12})\Psi_N^{\frac 12}.
  \end{align*}
The result \eqref{lem-4-7-est} follows by applying Cauchy--Schwarz.
\endproof

We can now combine our results to prove Theorem \ref{thm-deri-est}.
\begin{proof}[Proof of Theorem \ref{thm-deri-est}]
 The estimates in Lemmas \ref{lem-4-5}--\ref{lem-4-7} give the existence of a universal constant $C>0$ such that
\begin{align*}
  \left(\frac{\pt}{\pt t}-\Delta\right)&\Phi_N \leq  -\Psi_N+C\biggl(t\Phi_N^{\frac 12}+t^2\Phi_N+t^3\Phi_N^{\frac 32}+t^4\Phi_N^2\biggr)\Psi_N\nonumber\displaybreak[0]\\
  &\qquad\quad\!\!+C\Phi_N^{\frac 32}\biggl(1+t\Phi_N^{\frac 12}+t^2\Phi_N\biggr)\nonumber\\
  \leq & -\Psi_N+C\biggl(\left(1+t\Phi_N^{\frac 12}\right)^4-1\biggr)\Psi_N+C\Phi_N^{\frac 32}\left(1+t\Phi_N^{\frac 12}\right)^2.
\end{align*}
Let $\tau_N$ be the time
\begin{align*}
  \tau_N:= & \sup\{a\in [0,T_0]~|~t\Phi_N^{\frac 12}(x,t)\leq \left(C^{-1}+1\right)^{\frac 14}-1,\,\forall(x,t)\in \M\times [0,a]\}.
\end{align*}
Then on $\M\times [0,\tau_N]$, we have
\begin{align}\label{thm-deri-pf2}
  \left(\frac{\pt}{\pt t}-\Delta\right)&\Phi_N \leq C\Phi_N^{\frac 32}\left(C^{-1}+1\right)^{\frac 12}\leq  (1+C)\Phi_N^{\frac 32}.
\end{align}
Since the initial value of $\Phi_N$ is bounded by
\begin{align*}
  \Phi_N(x,0)\leq &~|Rm|^2(x,0)+|\nabla T|^2(x,0)+|\nabla^2\varphi|^2(x,0)\\
  \leq &~|Rm|^2(x,0)+|\nabla T|^2(x,0) 
  +2|\nabla T|^2|\psi|^2(x,0)+32|T|^4|\varphi|^2(x,0)\\
  \leq &~225\sup_{\M}\Lambda(\cdot,0)^2=225K_0^2,
\end{align*}
and $\M$ is compact, applying the maximum principle to \eqref{thm-deri-pf2} gives
\begin{equation}\label{thm-deri-pf3}
  \Phi_N(x,t)\leq \frac{225K_0^2}{(1-\frac 12(1+C)15K_0t)^2}
\end{equation}
on $\M\times [0,\tau_N]$. Let $\alpha>0$ be the universal constant
\begin{align*}
  \alpha:= & \frac 1{15}\left(\frac 12(1+C)+\left((C^{-1}+1)^{\frac 14}-1\right)^{-1}\right)^{-1}.
\end{align*}
 Denote
\begin{equation}\label{t*-eq}
  T_*=T_*(T_0,K_0):=\min\{T_0,\frac{\alpha}{K_0} \}>0.
\end{equation}
 By definition, $\tau_N\geq T_*$ for all $N\in \mathbb{N}$.  In particular,  \eqref{thm-deri-pf3} holds
on $\M\times [0,T_*]$ for all $N\in \mathbb{N}$. When $t\leq T_*$, the right hand side of \eqref{thm-deri-pf3} is bounded above by a positive constant $C_*$ depending only on $T_0$ and $K_0$.
\end{proof}

\subsection{Completing the proof}

Suppose  $\varphi(t), t \in [0,T_0]$ solves the Laplacian flow \eqref{Lap-flow-def} on a compact  manifold $\M$. Theorem \ref{thm-deri-est} implies that
\begin{equation*}
  t^{\frac k2}\left(|\nabla^kRm|(x,t)+|\nabla^{k+2}\varphi|(x,t)\right)\leq C_*(k+1)!
\end{equation*}
on $\M\times [0, T_*]$, where $T_*$ is given in \eqref{t*-eq}. Since $k+1\leq 2^k$, for any fixed $t\in (0,T_*]$, we have
\begin{equation*}
 |\nabla^kRm|(x,t)+|\nabla^{k+2}\varphi|(x,t)\leq Ck!r^{-k-2},
\end{equation*}
where $r=\sqrt{t}/2$ and $C=C_*T_*/4$ are uniform constants. Thus by Lemma \ref{lem-3-1} and the discussion following it, we conclude that $(\M,\varphi(t),g(t))$ is real analytic for each $t\in (0,T_*]$. Theorem \ref{mainthm-local} in the case when $U=M$ follows by iterating the above argument to cover the entire time interval $t\in (0,T_0]$.

\section{Local real analyticity}

In this section, we localize the discussion in \S \ref{sec:global} using a cut-off function to prove Theorem \ref{mainthm-local}. First, we show the existence of the required  function.

\begin{lem}\label{lem6-1}
Suppose $\varphi(t), t\in [0,T_0]$ is a smooth solution to the Laplacian flow \eqref{Lap-flow-def} on an open subset $U\subset \M$. Let $p\in U$ and $r>0$ so that $\overline{B}_{g(0)}(p,2r)\subset U$ is compact.  Let $\alpha>0$ be a constant and suppose that
 \begin{equation}\label{lem6-1-cond1}
\Lambda(x,t)=\left(|Rm|^2(x,t)+|\nabla T|^2(x,t)\right)^{\frac 12}\leq K
 \end{equation}
for all $(x,t)\in B_{g(0)}(p,2r)\times [0,T_*]$, where $0<T_*\leq \alpha/K$.

There exist constants $C_0=C_0(\alpha,r), C_1=C_1(\alpha,K,r)$, and a cut-off function $\eta: U\ra [0,1]$ with support in $B_{g(0)}(p,r)$, and with $\eta=1$ in $B_{g(0)}(p,r/2)$ such that
\begin{align}
  |\nabla\eta(x)|^2_{g(t)}~\leq &~C_0\eta(x),\label{lem6-1-cond2}\\
  -\Delta_{g(t)}\eta(x)~\leq & ~C_1\label{lem6-1-cond3}
\end{align}
on $U\times [0,T_*]$.
\end{lem}
\proof
Recall that along the Laplacian flow \eqref{Lap-flow-def}, the associated metric $g(t)$ evolves by \eqref{evl-g}, i.e.~with velocity $2h(t)$, where $h(t)$ is given in \eqref{h-tensor-1}.
 Let $\Theta(x,t)=tK^2|\nabla h(x,t)|^2$. Since
\begin{equation}\label{lem6-1-pf1-1}
  |Rm|\leq \Lambda(x,t)\leq K \quad \textrm{in}~B_{g(0)}(p,r)\times [0,T_*],
\end{equation}
by a straightforward adjustment to the proof of \cite[Lemma 14.3]{Chow-etal-2008},   there exists a cut-off function $\eta:U\ra [0,1]$ with support in $B_{g(0)}(p,r)$ and with $\eta=1$ in $B_{g(0)}(p,r/2)$ such that
\begin{align}
  |\nabla\eta(x)|^2_{g(t)}\leq & ~\frac {C(\alpha)}{r^2}\eta(x), \label{lem6-1-pf2}\\
  -\Delta_{g(t)}\eta(x) \leq &~\frac{C(\alpha,\sqrt{K}r)}{r^2}+\frac{C(\alpha)}{K^{\frac 32}r}\sup_{s\in[0,t]}(\eta\Theta)^{\frac 12}(x,s)\label{lem6-1-pf3}
\end{align}
for some constants $C(\alpha)$ and $C(\alpha,\sqrt{K}r)$, for all $(x,t)\in B_{g(0)}(p,r)\times [0,T_*]$. We obtain \eqref{lem6-1-cond2} from \eqref{lem6-1-pf2} by defining $C_0=C(\alpha,n)/{r^2}$.

The cut-off function $\eta$ here is constructed by a composition of a scalar function with the Riemannian distance function $d_{g(0)}(x,p)$ with respect to the initial metric $g(0)$. The key is that the bound \eqref{lem6-1-pf1-1} and the fact $|T|^2=-R$ imply that $h$ in \eqref{h-tensor-1} is uniformly bounded in $B_{g(0)}(p,r)\times [0,T_*]$, which in turn implies the uniform equivalence of the metrics $g(t)$ for $t\in[0,T_*]$.

We next show \eqref{lem6-1-cond3}. Under the assumption \eqref{lem6-1-cond1}, the local Shi-type derivative estimates from \cite[Theorem 4.3]{Lotay-Wei-1} for $Rm$ and $T$ give that
\begin{equation}\label{lem6-1-pf4}
  t^{\frac 12}\left(|\nabla Rm|(x,t)+|\nabla^2T|(x,t)\right)\leq C(\alpha,\sqrt{K}r)K,
\end{equation}
for a constant $C(\alpha,\sqrt{K}r)$, for all $(x,t)\in B_{g(0)}(p,r)\times [0,T_*]$. 
(Note that in \cite[Theorem 4.3]{Lotay-Wei-1}, we only state the estimate when $\alpha=1$, but a trivial adjustment of the proof gives \eqref{lem6-1-pf4} as stated.) We deduce that
\begin{align}\label{lem6-1-pf5}
  \Theta(x,t)= &~tK^2|\nabla h(x,t)|^2\nonumber\\
  \leq &~CtK^2\left(|\nabla Rm|^2(x,t)+|T|^2|\nabla T|^2(x,t)\right)\nonumber\displaybreak[0]\\
   \leq &~CtK^2\left(|\nabla Rm|^2(x,t)+|Rm|^3(x,t)\right)\nonumber\\
   \leq &~C\left(C(\alpha,\sqrt{K}r)^2+\alpha\right)K^4
\end{align}
for all $(x,t)\in B_{g(0)}(p,r)\times [0,T_*]$, where in the third inequality we used Propositions \ref{Ric-prop}--\ref{prop-nabla-T} and $tK\leq \alpha$. Combining \eqref{lem6-1-pf5} and \eqref{lem6-1-pf3} gives \eqref{lem6-1-pf4}.
\endproof

\begin{rem}
Although $\eta$ in Lemma \ref{lem6-1} is constructed by a composition with a Riemannian distance function, by the standard Calabi's trick (see, for example, \cite[pp.453--456]{Chow-etal-2007}), we can for the purpose of applying the maximum principle treat $\eta$ as a smooth function.
\end{rem}

\begin{thm}\label{thm6-1}
Suppose that $\varphi(t), t\in [0,T_0]$ solves
the Laplacian flow \eqref{Lap-flow-def} on an open set $U\subset \M$. Let $p\in U$ and $r>0$ be such that  $\overline{B_{g(0)}(p,2r)}\subset U$ is compact. Let $K=\sup_{B_{g(0)}(p,2r)\times [0,T_0]}\Lambda(x,t)$, where $\Lambda(x,t)$ is given in \eqref{lambda-eq}, and let $\alpha>0$ be such that $\alpha\leq KT_0$.

There exist positive constants $L, C,T_*$ depending only on $\alpha,K,r$ such that
\begin{equation}\label{thm6-1-est}
  t^{\frac k2}\left(|\nabla^kRm|(x,t)+|\nabla^{k+1}T|(x,t)+|\nabla^{k+2}\varphi|(x,t)\right)\leq C L^{\frac k2}(k+1)!
\end{equation}
for all $k\in \mathbb{N}$ and $(x,t)\in B_{g(0)}(p,r/2)\times [0,T_*]$.
\end{thm}

\begin{rem}
If $\varphi(t), t\in [0,T_0]$ is a smooth solution to the Laplacian flow \eqref{Lap-flow-def} on an open set $U\subset \M$, then a similar argument as in the proof of \cite[Theorem 1.3]{Lotay-Wei-1} shows that $\Lambda(x,t)$ is bounded on $U\times [0,T_0]$. Thus, $K=\sup_{B_{g(0)}(p,2r)\times [0,T_0]}\Lambda(x,t)$ in Theorem \ref{thm6-1} is finite.
\end{rem}

\proof
We consider localized modifications of $\Phi_N, \Psi_N$ in \S \ref{sec:global}, in a similar spirit to \cite{kot-2013}. For $L>0$, to be determined later, we define
\begin{align*}
  \alpha_k&=\frac{\eta^{\frac{k+1}2}}{L^{\frac k2}}a_k, &  \beta_k&=\frac{\eta^{\frac{k+1}2}}{L^{\frac k2}}b_k, & \gamma_k&=\frac{\eta^{\frac{k+1}2}}{L^{\frac k2}}c_k, & \textrm{for } k\geq 0, \\
  \tilde{\alpha}_k&=\frac{\eta^{\frac{k}2}}{L^{\frac{k-1}2}}\tilde{a}_k, &  \tilde{\beta}_k&=\frac{\eta^{\frac{k}2}}{L^{\frac{k-1}2}}\tilde{b}_k, & \tilde{\gamma}_k&=\frac{\eta^{\frac{k}2}}{L^{\frac{k-1}2}}\tilde{c}_k, &\textrm{for } k\geq 1,
\end{align*}
where $a_k,b_k,c_k,\tilde{a}_k,\tilde{b}_k,\tilde{c}_k$ are defined in \eqref{def-ak}--\eqref{def-ak-tilde} and $\eta$ is the cut-off function constructed in Lemma \ref{lem6-1}. We further define
\begin{equation}\label{def-Phi-td}
  \tilde{\Phi}_N=\sum_{k=0}^N\left(\alpha_k^2+\beta_k^2+\gamma_k^2\right)\quad
  \text{and}\quad \tilde{\Psi}_N=\sum_{k=1}^N\left(\tilde{\alpha}_k^2+\tilde{\beta}_k^2+\tilde{\gamma}_k^2\right).
\end{equation}
We aim to estimate $\tilde{\Phi}_N$. We first compute an evolution inequality for $\tilde{\Phi}_N$, by looking at each of $\alpha_k^2$, $\beta_k^2$ and $\gamma_k^2$ in turn.  First,
\begin{align}\label{6-1}
  \left(\frac{\pt}{\pt t}-\Delta\right)\alpha_k^2&=  \frac{\eta^{k+1}}{L^{k}}\left(\frac{\pt}{\pt t}-\Delta\right)a_k^2+\frac{a_k^2}{L^{k}}\left(\frac{\pt}{\pt t}-\Delta\right)\eta^{k+1}\nonumber\\
  &\quad-\frac 2{L^{k}}\langle\nabla\eta^{k+1},\nabla a_k^2\rangle.
\end{align}
By \eqref{lem6-1-cond3} and $t\leq T_0$, the second term on the right hand side of \eqref{6-1} satisfies
\begin{align}\label{6-2}
  \frac{a_k^2}{L^{k}}\left(\frac{\pt}{\pt t}-\Delta\right)\eta^{k+1}= & ~ \frac{a_k^2}{L^{k}}\left(-k(k+1)\eta^{k-1}|\nabla\eta|^2-(k+1)\eta^k\Delta\eta\right)\nonumber\displaybreak[0]\\
  \leq & ~(k+1)C_1\frac{a_k^2}{L^{k}} \eta^k=\frac{C_1t}{L(k+1)}\tilde{\alpha}_k^2  \leq
  \frac{C_1T_0}{L(k+1)}\tilde{\alpha}_k^2
\end{align}
on $U\times [0,\alpha/K]$.  To estimate the third term of \eqref{6-1}, we use \eqref{lem6-1-cond2}, $t\leq T_0$ and the Cauchy--Schwarz inequality:
\begin{align}\label{6-3}
  -\frac 2{L^{k}}\langle\nabla\eta^{k+1},\nabla a_k^2\rangle \leq& ~ 4\frac {\eta^k|\nabla\eta|}{L^k}\frac{t^k|\nabla^kRm||\nabla^{k+1}Rm|}{k!(k+1)!} \leq 
 \frac 14\tilde{\alpha}_{k+1}^2+\frac{16C_0T_0}L\tilde{\alpha}_k^2
\end{align}
on $U\times [0,\alpha/K]$. Substituting \eqref{lem-4-5-pf2}, \eqref{6-2} and \eqref{6-3} into \eqref{6-1}, we have
\begin{align}\label{6-4}
  \left(\frac{\pt}{\pt t}-\Delta\right)\alpha_k^2  \leq &~-\frac 74 \tilde{\alpha}_{k+1}^2+\frac {C}L\tilde{\alpha}_k^2+\frac{\eta^{k+1}}{L^{k}}(I_1(k)+I_2(k)+I_3(k)),
\end{align}
where $C=C(\alpha,r,T_0)$. Using \eqref{lem-4-5I1}--\eqref{lem-4-5I3}, we can estimate
\begin{align}\label{6-5}
  \sum_{k=1}^N\frac{\eta^{k+1}}{L^{k}}(I_1(k)&+I_2(k)+I_3(k))
   \leq  CK\tilde{\Phi}_N+CK^2\tilde{\Phi}_N^{\frac 12}+C\frac tLK\tilde{\Psi}_N   \nonumber\\
   &+C\frac tL\tilde{\Phi}_N^{\frac 12}\tilde{\Psi}_N
   +CK^{\frac 12}\tilde{\Psi}_{N+1}^{\frac 12}\tilde{\Phi}_N^{\frac 12}+C\frac{t^2}{L^2}\tilde{\Phi}_N\tilde{\Psi}_N.
\end{align}
Combining \eqref{lem-4-5-pf3} and \eqref{6-1}--\eqref{6-5}, we obtain
\begin{align}\label{6-6}
    \left(\frac{\pt}{\pt t}-\Delta\right)\sum_{k=0}^N\alpha_k^2 \leq & -\frac 32\tilde{\Psi}_{N+1}+C(\tilde{\Phi}_N+1)\nonumber \\
&~+C\left(\frac 1L+\frac tL\tilde{\Phi}_N^{\frac 12}+\frac{t^2}{L^2}\tilde{\Phi}_N \right) \tilde{\Psi}_N
\end{align}
on $U\times [0,\alpha/K]$, again using Cauchy--Schwarz and $tK\leq \alpha$, where the constant $C$ depends only on $\alpha,r,K,T_0$. We can deal with $\sum_{k=0}^N\beta_k^2$ and $\sum_{k=0}^N\gamma_k^2$ similarly and obtain the following estimate:
\begin{align}\label{6-7}
    \left(\frac{\pt}{\pt t}-\Delta\right)&\sum_{k=0}^N(\beta_k^2+\gamma_k^2) \leq  -\frac 32\tilde{\Psi}_{N+1}+C(\tilde{\Phi}_N+1)\nonumber \\
&~+C\left(\frac 1L+\frac tL\tilde{\Phi}_N^{\frac 12}+\frac{t^2}{L^2}\tilde{\Phi}_N+\frac{t^4}{L^4}\tilde{\Phi}_N^4\right) \tilde{\Psi}_N.
\end{align}
Since $-\tilde{\Psi}_{N+1}\leq -\tilde{\Psi}_{N}$, we may put together the estimates \eqref{6-6}--\eqref{6-7} and choose $L$ large enough so that
\begin{align*}
    \left(\frac{\pt}{\pt t}-\Delta\right)\tilde{\Phi}_N \leq &~ -\tilde{\Psi}_{N}+C_2(\tilde{\Phi}_N+1)+C_3\left(\left(1+t\tilde{\Phi}_N^{\frac 12}\right)^4-1\right)\tilde{\Psi}_{N},
\end{align*}
where $C_2,C_3$ depend only on $\alpha,r,K,T_0$. Let
\begin{equation*}
  \tau_N=\{a\in [0,\alpha/K]~|~ t\tilde{\Phi}_N^{\frac 12}\leq (C_3^{-1}+1)^{\frac 14}-1 ,\forall (x,t)\in U\times [0,a]\}.
\end{equation*}
Then on $B_{g(0)}(p,r)\times [0,\tau_N]$,
\begin{equation*}
  \left(\frac{\pt}{\pt t}-\Delta\right)\tilde{\Phi}_N \leq C_2(\tilde{\Phi}_N+1).
\end{equation*}
As $\tilde{\Phi}_N=0$ on $\pt B_{g(0)}(p,r)\times [0,\tau_N]$ and $\tilde{\Phi}_N(\cdot,0)\leq 225K^2$, the maximum principle gives that
\begin{equation*}
  \tilde{\Phi}_N\leq (225K^2+1)e^{C_2T_0}:=C_4
\end{equation*}
for all $(x,t)\in B_{g(0)}(p,r)\times [0,\tau_N]$. Let
\begin{equation*}
  T_*=\min\{\alpha/K, C_4^{-\frac 12}((C_3^{-1}+1)^{\frac 14}-1)\}>0,
\end{equation*}
which depends only on $\alpha,r,K,T_0$. Then $\tau_N\geq T_*$ for all $N\in \mathbb{N}$ and thus
  $\tilde{\Phi}_N\leq C_4$
for all $N\in \mathbb{N}$ and $(x,t)\in B_{g(0)}(p,r)\times [0,T_*]$. Since $\eta\equiv 1$ on $B_{g(0)}(p,r/2)$, the estimate \eqref{thm6-1-est} follows.
\endproof

Finally, Theorem \ref{mainthm-local} follows from Theorem \ref{thm6-1} and the discussion in \S \ref{sec:real-anal}.

\bibliographystyle{Plain}

\end{document}